\documentclass[11pt]{amsart}

\usepackage[T1]{fontenc}
\usepackage[utf8]{inputenc}

\usepackage[margin=3cm]{geometry}
\usepackage{microtype}

\usepackage{amsmath,amssymb,amsfonts,amsthm}
\usepackage{mathtools}
\usepackage{bm}
\numberwithin{equation}{section}

\usepackage[colorlinks=true,linkcolor=blue,citecolor=blue,urlcolor=blue]{hyperref}

\theoremstyle{plain}
\newtheorem{theorem}{Theorem}[section]
\newtheorem{proposition}[theorem]{Proposition}
\newtheorem{lemma}[theorem]{Lemma}
\newtheorem{corollary}[theorem]{Corollary}

\theoremstyle{definition}
\newtheorem{definition}[theorem]{Definition}

\theoremstyle{remark}
\newtheorem{remark}[theorem]{Remark}

\title[Discontinuous fluxes without non-degeneracy]
{Multidimensional scalar conservation laws with discontinuous flux:
well-posedness without non-degeneracy}

\author{Darko Mitrovi\'c}
\address{Faculty of Sciences and Mathematics, University of Montenegro, Podgorica, Montenegro}
\email{darko.mitrovic.mne@gmail.com}


\begin{document}

\begin{abstract}
We establish existence, uniqueness, and local \(L^1\)-stability for
multidimensional scalar conservation laws with discontinuous heterogeneous
flux, without imposing a non-degeneracy condition on the physical
space--time flux. The discontinuity set may be a locally finite family of
\(C^2\) hypersurfaces whose intersections and singular points are contained
in a closed set of vanishing \((d-1)\)-dimensional Hausdorff measure.

The main obstacle is the possible failure of multidimensional compactness
when a nontrivial linear combination of the space--time flux components is
constant on an interval of states. We show that, for interface problems, this
obstruction can be reduced to the physical normal flux. Its flat intervals
are collapsed by a normal-flux quotient that preserves the normal flux and
hence the Rankine--Hugoniot relation. On the non-flat region, the time and
tangential components are replaced by auxiliary polynomial functions of the
normal flux, while the physical normal component is left unchanged. The
resulting auxiliary space--time vector satisfies the interval
non-degeneracy required by Panov's compactness theory. This yields strong
compactness and strong one-sided traces of the quotient variables.

The interface admissibility condition is obtained by projecting the
vanishing-viscosity germ associated with the two one-sided physical normal
fluxes. A localized viscous comparison shows that the quotient traces belong
to this maximal \(L^1\)-dissipative germ, and a Young-measure contraction
argument recovers strong convergence of the physical states. Curved
interfaces are treated by flattening, localization, and finite-propagation
patching. The theory is first constructed for \(BV\) initial data and then
extended by continuity to
\(L^1_{\mathrm{loc}}\cap L^\infty\) data taking values in the invariant
interval.
\end{abstract}

\maketitle

\section{Introduction}

We study multidimensional scalar conservation laws with spatially
discontinuous heterogeneous fluxes,
\begin{equation}\label{eq:intro-main}
    \partial_t u+\operatorname{div}_x F(x,u)=0,
    \qquad
    u(0,x)=u_0(x),
    \qquad
    (t,x)\in(0,\infty)\times\mathbb R^d.
\end{equation}
The initial datum and the solution are assumed to take values in the fixed
compact invariant interval
\[
    K=[-a,a].
\]

At every regular point of the discontinuity set, the ordered pair of
one-sided normal fluxes determines a vanishing-viscosity germ in the sense of
the \(L^1\)-dissipative germ theory of \cite{AKR,AM}. We call a solution
admissible if, locally near every regular interface point, it belongs to the
solution class associated with this germ. The precise formulation, including
the use of normal-flux quotient variables when the normal flux has flat
intervals, is given in Section~\ref{sec:general-interface}.

Our main result is the following.

\begin{theorem}[Main theorem]
\label{thm:intro-main}
Let the discontinuity set of \(F\) be decomposed as
\[
    \Gamma
    =
    \Gamma_{\mathrm{reg}}
    \mathbin{\dot\cup}
    \Gamma_p,
\]
where \(\Gamma_p\) is closed and
\[
    \mathcal H^{d-1}(\Gamma_p)=0.
\]
Assume that every \(x_0\in\Gamma_{\mathrm{reg}}\) has a neighbourhood
\(U_{x_0}\) in which, after a permutation of the spatial coordinates,
\[
    \Gamma\cap U_{x_0}
    =
    \{x_1=\zeta_{x_0}(x')\}\cap U_{x_0},
    \qquad
    \zeta_{x_0}\in C^2,
\]
and the flux admits the two-branch representation
\begin{equation}\label{eq:intro-local-flux}
\begin{aligned}
    F(x,\lambda)
    &=
    H\bigl(\zeta_{x_0}(x')-x_1\bigr)
    f_{x_0}^-(x,\lambda)+
    H\bigl(x_1-\zeta_{x_0}(x')\bigr)
    f_{x_0}^+(x,\lambda),
    \ \
    (x,\lambda)\in U_{x_0}\times K,
\end{aligned}
\end{equation}
where
\[
    f_{x_0}^\pm
    \in
    C^2(U_{x_0}\times K;\mathbb R^d).
\]
Assume moreover that
\[
    f_{x_0}^\pm(x,-a)
    =
    f_{x_0}^\pm(x,a)
    =
    0,
    \qquad x\in U_{x_0}.
\]

Then, for every
\[
    u_0\in BV(\mathbb R^d;K),
\]
problem \eqref{eq:intro-main} admits a unique solution in the admissibility
class associated with the vanishing-viscosity germs of the corresponding
one-sided normal fluxes, in the sense of \cite{AKR,AM}.

Moreover, if \(u\) and \(v\) are admissible solutions with initial data
\(u_0\) and \(v_0\), respectively, then for every \(T,R>0\) there exist
\(\overline R>R\) and \(C_{T,R}>0\) such that
\begin{equation}\label{eq:intro-stability}
    \operatorname*{ess\,sup}_{t\in(0,T)}
    \int_{B_R}|u(t,x)-v(t,x)|\,dx
    \leq
    C_{T,R}
    \int_{B_{\overline R}}
        |u_0(x)-v_0(x)|\,dx.
\end{equation}
Consequently, the solution map extends uniquely by local
\(L^1\)-continuity to all initial data
\[
    u_0
    \in
    L^1_{\mathrm{loc}}(\mathbb R^d)
    \cap L^\infty(\mathbb R^d),
    \qquad
    u_0(x)\in K
    \quad\text{for a.e. }x.
\]
\end{theorem}

Theorem~\ref{thm:intro-main} closes the principal gap left by the existing
multidimensional theory: it provides a well-posedness framework based on the
vanishing-viscosity germ for discontinuity sets that are locally smooth graphs
outside an \(\mathcal H^{d-1}\)-negligible set, without any
genuine-nonlinearity or multidimensional non-degeneracy assumption on the
physical flux.

The proof combines a new compactness mechanism for the flat-interface problem
with the geometric localization procedure developed in
\cite{Mitrovic2025Nondegenerate}. We first consider the interface
\[
    \Gamma=\{x_1=0\}
\]
and \(BV\) initial data. The pure-viscosity approximation provides uniform
local Radon-measure bounds for the time derivative and the tangential spatial
derivatives, but not for the normal derivative, where the viscous interface
layer is concentrated. These estimates allow us to replace the time and
tangential components of the auxiliary space--time flux by polynomial
functions of the physical normal flux \(N\), while leaving the normal
component unchanged. The resulting vector
\[
    \mathcal N_N
    =
    \bigl(
        N^2,N,N^3,\ldots,N^{d+1}
    \bigr)
\]
satisfies Panov's non-degeneracy condition on the regions where
\(\partial_\lambda N\neq0\).

Flat intervals of the normal flux are collapsed by the quotient
\[
    \Pi_N(x,\lambda)
    =
    \int_{-a}^{\lambda}
        |\partial_sN(x,s)|^2\,ds.
\]
Localized compactness and trace results yield strong convergence and strong
one-sided traces of the quotient variables. The vanishing-viscosity germ of
the two physical normal fluxes is then projected through the corresponding
quotient maps, and the viscous approximation is shown to select this projected
germ. Its \(L^1\)-dissipativity, combined with a Young-measure contraction
argument, gives strong convergence of the physical states and uniqueness.
For a curved interface, we flatten the hypersurface locally and apply the
same construction to the transformed co-normal flux. The resulting local
solutions are extended, compared in smaller space--time cones, and patched
together by the finite-propagation argument of
\cite{Mitrovic2025Nondegenerate}; the negligible interaction set is finally
removed by a \(W^{1,1}\)-cutoff argument.

Equations such as \eqref{eq:intro-main} arise naturally in models where the
medium, geometry, or constitutive law changes abruptly in space. Typical
examples include flow in heterogeneous porous media, sedimentation and
clarification models, traffic flow with discontinuous road conditions, locally
constrained flows, network-type models, and conservation laws in composite
materials. In such problems the flux is not a smooth function of the spatial
variable, but changes across interfaces separating different regimes. The
discontinuity of the flux is therefore not a technical artifact: it encodes a
physical change of medium, capacity, permeability, or admissible flow.

In several space dimensions, the existence theory for discontinuous-flux
problems is commonly based on compactness obtained through velocity averaging (see e.g. \cite{AleksicMitrovic2009,Mitrovic2025Nondegenerate,Panov_arma}).
This approach requires a non-degeneracy condition on the full space--time
transport field. More precisely, for every nonzero vector
\[
    \xi=(\xi_0,\xi_1,\ldots,\xi_d)\in\mathbb R^{d+1},
\]
the function
\[
    \lambda\longmapsto
    \xi_0\lambda+\sum_{j=1}^d\xi_jF_j(x,\lambda)
\]
is required not to be constant on any non-degenerate interval of the state
variable. This condition may fail because of the time component, a tangential
flux component, or a nontrivial linear combination of several components, and
therefore excludes many physically natural fluxes.

The purpose of the present paper is to overcome this restriction for
discontinuous-flux interface problems. At an interface, the
Rankine--Hugoniot relation involves only the normal component of the physical
flux. This allows us to keep the normal flux unchanged, while replacing the
time and tangential components by auxiliary components used only in the
velocity-averaging argument. These components are chosen so that the required
non-degeneracy is restored on the non-flat part of the normal flux. Flat
intervals of the normal flux are not excluded. Instead, the corresponding
states are identified through a quotient projection which preserves the
physical normal flux and hence the Rankine--Hugoniot relation.

The problem has a long history. Early scalar works with explicitly
discontinuous flux functions include the one-dimensional Cauchy-problem
analysis of \cite{GimseRisebro1992} and Diehl's sedimentation models
\cite{Diehl1995,DiehlCMP1996}. These works helped establish scalar
conservation laws with discontinuous flux as a separate object of study.
Physical models, introductory overviews, and network or constrained-flow
motivations are treated, for instance, in
\cite{AndreianovDiscConstraint2010,Buerger2008,ColomboGoatin2007,GNPT2007}.

Several complementary approaches have since developed. One important line is
based on adapted entropy inequalities, interface admissibility conditions,
germs, and \(L^1\)-dissipative solvers. This includes
\cite{AdimurthiJHDE2005,AdimurthiNHM2007,AKR-NHM,AKR,AM,AudussePerthame2005}.
In particular, the multidimensional uniqueness result from \cite{AM} applies
to special interface configurations which are mutually disjoint and locally
representable as graphs.

Kinetic formulations and process-solution techniques provide another important
approach. The kinetic framework for scalar conservation laws goes back, in
this context, to \cite{LPT94}, while discontinuous-flux problems treated by
kinetic or process-solution methods include
\cite{BachmannVovelle2006,CrastaDCDP2015,CrastaDCDPG2016}.

Compactness, non-degeneracy, and trace methods form a further class of
results. Compactness mechanisms related to non-degeneracy and velocity
averaging appear in
\cite{AleksicMitrovic2009,KarlsenRascleTadmor2007,
Mitrovic2025Nondegenerate,Panov_arma,panov99}.
Strong trace results for generalized solutions and quasi-solutions are
developed in
\cite{pan_traces,PanovTraces2007}.
Treatment of entropy solutions for discontinuous fluxes, including an
Audusse--Perthame type admissibility structure in a broader setting, is given
in \cite{KM-MH,Panov2009,PiccoliTournus2018}.

Approximation and numerical schemes constitute another strand of the theory.
Godunov-type and finite-volume approximations, as well as vanishing-viscosity
selection mechanisms, are treated in
\cite{AdimurthiJNA2004,AKR-NHM,AndreianovDiscConstraint2010,GhoshalJanaTowers2020}.

Finally, \(BV\), total variation, and \(BV\)-spatial-flux assumptions appear
in special structural situations. The classical background is \cite{Volpert1967}.
For discontinuous-flux results involving total variation bounds,
\(BV\)-type hypotheses, or \(BV\)-based constructions, see
\cite{AdimurthiGhoshalGowda2011,GhoshalJanaTowers2020,PiccoliTournus2018}.
When \(BV\)-bounds are not available, fractional or higher regularity may still
persist under additional structural assumptions; see
\cite{GhoshalJuncaParmar2024NARWA,GhoshalJuncaParmar2024SIMA,EKM2}.
Related recent work on viscosity approximation and non-aligned or non-crossing
discontinuous-flux configurations includes
\cite{KarlsenMitrovicNedeljkov2025,Zajmovic2025}, where the appearance of
singular measures in the vanishing-viscosity limit is observed.

We next explain how the present construction combines and modifies several
ideas from the existing theory. For spatially homogeneous scalar conservation
laws
\[
    \partial_t u+\operatorname{div}_x f(u)=0,
\]
Kruzhkov's theory gives existence, uniqueness, and \(L^1\)-stability of entropy
solutions under the usual regularity assumptions on the flux \cite{Kru}. For
heterogeneous fluxes \(F=F(x,u)\), additional regularity in \(x\) is needed in
order to make the entropy source terms meaningful. The discontinuous-flux case
is fundamentally different, because the one-sided Kruzhkov inequalities away
from the discontinuity set do not determine the admissible interface
connection.

In one space dimension, this missing interface condition has been studied
through adapted entropies, interface connections, optimal entropy solutions,
vanishing-viscosity selection, and \(L^1\)-dissipative germs; see, for example,
\cite{AdimurthiJHDE2005,AKR-NHM,AKR,AudussePerthame2005}. The flat-part
reduction developed in \cite{Mitrovic2011} is particularly close to the
present construction. There, values of the solution lying on intervals where
the flux is constant can be replaced without changing the physical flux. The
present paper extends this idea to multidimensional discontinuous fluxes. The
extension is not direct: in an interface problem only the normal flux is
physically fixed by the Rankine--Hugoniot relation, while the time and
tangential flux components may be replaced in the compactness argument.

Already in the one-dimensional single-interface problem, the absence of a
non-degeneracy condition creates a genuine compactness difficulty. Consider,
for instance,
\begin{equation}\label{eq:intro-one-dimensional-interface}
    \partial_t u
    +
    \partial_x
    \bigl(
        f(x,u)H(-x)+g(x,u)H(x)
    \bigr)
    =
    0,
    \qquad
    u(0,x)=u_0(x),
\end{equation}
where
\[
    u_0\in BV(\mathbb R),
    \qquad
    -a\leq u_0\leq a,
\]
and
\[
    f(x,\pm a)=g(x,\pm a)=0 .
\]
The last condition yields the invariant region
\[
    -a\leq u\leq a .
\]

The purpose of \cite{Mitrovic2011} was precisely to remove the
non-degeneracy assumption in this one-dimensional setting. The obstruction
occurs on intervals on which one of the flux branches is constant. On such an
interval, the physical flux does not distinguish between different values of
the state variable. These values may therefore be replaced by a single
representative, for example by projecting them onto the lower endpoint of the
corresponding flat interval; see
\cite[Lemma~2.3, formula~(27)]{Mitrovic2011}.

After the flat parts are collapsed, strong compactness can be recovered for
the projected representatives by the precompactness theory of
\cite{Panov_arma}. The corresponding strong one-sided traces are obtained
from the trace theory for generalized solutions and quasi-solutions developed
in \cite{pan_traces,PanovTraces2007}. Thus, the one-dimensional theory of
\cite{Mitrovic2011} already permits flat intervals of the physical flux and
does not impose the usual non-degeneracy condition.

The essential limitation of \cite{Mitrovic2011} lies elsewhere. Its
uniqueness argument uses a transformation of the unknown together with a
crossing-type condition, in the spirit of \cite{KRT}. This construction is
essentially adapted to a single discontinuity point and does not provide a
local interface condition that can be assigned independently to each member
of a multidimensional family of discontinuity interfaces.

The later theory of \(L^1\)-dissipative solvers and interface germs
\cite{AKR}, together with its multidimensional extension to suitable
graph-type interface configurations in \cite{AM}, provides the appropriate
local uniqueness mechanism. The present paper combines this germ formulation
with the flat-part projection introduced in \cite{Mitrovic2011}. The
projection removes the non-degeneracy obstruction, while the germ provides a
local admissibility condition that can be imposed at each regular point of a
multidimensional discontinuity interface. These ingredients alone, however,
do not yield existence for more general interface geometries. New localization
and approximation ideas are needed even for a single closed smooth interface,
such as a sphere, which is only locally, but not globally, representable as a
graph.

The existence problem for general multidimensional interface geometries was
only recently addressed in \cite{Mitrovic2025Nondegenerate}. The main new idea
there was a localized vanishing-viscosity construction. Near each regular
point of the discontinuity set, the interface is flattened and the
vanishing-viscosity selection is analyzed in the resulting local
flat-interface problem. The local solutions are then shown to be compatible
on overlapping charts and are patched together by means of local
\(L^1\)-stability and finite propagation. This made it possible to treat
smooth interfaces which are only locally representable as graphs, including
closed hypersurfaces such as spheres, as well as more general geometrically
regular interface configurations. The compactness argument in that work,
however, still required a non-degeneracy condition on the relevant
multidimensional space--time flux.

The present paper removes this remaining restriction. It combines the
flat-part projection introduced in \cite{Mitrovic2011}, the
\(L^1\)-dissipative interface-germ theory of \cite{AKR,AM}, and the localized
vanishing-viscosity construction of
\cite{Mitrovic2025Nondegenerate}. After the interface is flattened, the
physical co-normal flux is kept unchanged, while the time and tangential
components are replaced by auxiliary components used only in the compactness
argument. The velocity-averaging non-degeneracy condition is therefore imposed
only on these auxiliary space--time vectors and only on the open non-flat
part of the co-normal flux. Flat intervals of the co-normal flux are allowed
and are collapsed by the associated quotient projection.

The resulting theory applies to a general smooth discontinuity interface and,
more generally, to a locally finite family of smooth interfaces whose mutual
intersections are confined to a set of vanishing
\((d-1)\)-dimensional Hausdorff measure. Thus the method covers geometrically
regular interface configurations for which the admissibility condition can be
formulated and selected locally at almost every point of the discontinuity
set.

We now describe the construction in the model flat-interface case
\begin{equation}\label{eq:intro-flat-flux}
    F(x,\lambda)
    =
    H(-x_1)f(x,\lambda)+H(x_1)g(x,\lambda),
    \qquad
    \Gamma=\{x_1=0\},
\end{equation}
where \(f\) and \(g\) denote the left and right heterogeneous flux branches. We
impose the invariant-region condition
\begin{equation}\label{eq:intro-invariant}
    f(x,\pm a)=0,
    \qquad
    g(x,\pm a)=0,
\end{equation}
which yields the maximum principle
\[
    -a\leq u\leq a .
\]

Across the interface, the conservation law sees only the normal flux. Hence any
weak solution must satisfy a Rankine--Hugoniot relation involving the one-sided
normal components
\[
    f_1(0^-,x',\lambda),
    \qquad
    g_1(0^+,x',\lambda).
\]
This relation alone, however, does not determine which trace pair is
admissible. The usual Kruzhkov inequalities on the two sides of the interface
are therefore not sufficient for uniqueness. A separate interface
admissibility condition is needed. In this paper the admissibility condition is
derived from a localized vanishing-viscosity procedure.

In the model flat-interface case
\[
    \Gamma=\{x_1=0\},
\]
the physically relevant fluxes at the interface are the one-sided normal
components
\[
    \lambda\longmapsto f_1(x,\lambda)
    \quad\text{and}\quad
    \lambda\longmapsto g_1(x,\lambda).
\]
Accordingly, the compactness construction developed below is organized around
the non-flat regions of these normal fluxes. The time and tangential components
may be replaced by auxiliary components, whereas the normal components are
kept unchanged in order to preserve the Rankine--Hugoniot relation.

The replacement is asymmetric. The time and tangential components are used only
to obtain compactness and may therefore be modified. The normal component,
however, is constrained by conservation across the interface and by the
Rankine--Hugoniot relation. It is therefore kept equal to the physical normal
flux throughout the argument.

The viscous approximation provides local bounds on
\[
    \partial_t u^\varepsilon
    \qquad\text{and}\qquad
    \partial_{x_j}u^\varepsilon,\quad j=2,\ldots,d,
\]
but not on the normal derivative \(\partial_{x_1}u^\varepsilon\). The missing
normal estimate is precisely where the interface layer is concentrated.
Consequently, the time and tangential components of the space--time flux may
be replaced by locally Lipschitz functions of the unknown, while the normal
component is kept equal to the physical normal flux.

Thus we introduce auxiliary space--time vectors of the form
\[
    \mathcal F_-(x,\lambda)
    =
    \bigl(
        h_{-,0}(x,\lambda),
        f_1(x,\lambda),
        h_{-,2}(x,\lambda),
        \ldots,
        h_{-,d}(x,\lambda)
    \bigr)
\]
on the left and
\[
    \mathcal F_+(x,\lambda)
    =
    \bigl(
        h_{+,0}(x,\lambda),
        g_1(x,\lambda),
        h_{+,2}(x,\lambda),
        \ldots,
        h_{+,d}(x,\lambda)
    \bigr)
\]
on the right. These vectors are auxiliary compactness fluxes; they are not the
physical fluxes of the conservation law. The auxiliary components are chosen so
that the replaced vectors are non-degenerate on every interval on which the
corresponding normal flux is not constant. This non-degeneracy is not imposed
on the physical flux; it is only a compactness device made possible by the time
and tangential estimates.

It remains to treat intervals on which the normal flux itself is flat.  Such
intervals are not excluded.  Let
\[
    N(x,\lambda)
\]
denote the normal flux, for instance \(N=f_1\) on the left or \(N=g_1\) on
the right, and assume that \(N\) is \(C^1\) in \((x,\lambda)\).  Set
\[
    D_N(x,\lambda):=\partial_\lambda N(x,\lambda)
\]
and define
\begin{equation}
\label{eq:intro-nonflat-region}
    A_N
    :=
    \bigl\{
        (x,\lambda)\in\Omega\times K:
        D_N(x,\lambda)\neq0
    \bigr\}.
\end{equation}
Since \(D_N\) is continuous, \(A_N\) is open.  On every interval compactly
contained in \((A_N)_x\), the derivative \(D_N(x,\cdot)\) has a fixed
nonzero sign, and hence \(N(x,\cdot)\) is strictly monotone.

On \(A_N\), we replace the time and tangential components of the auxiliary
space--time flux by suitable polynomial functions of \(N\), while keeping
the physical normal component unchanged.  The resulting kinetic velocity
satisfies the qualitative non-degeneracy condition required by Panov's
compactness theory.  Indeed, for every nonzero direction \(\xi\), its scalar
product with \(\xi\) cannot vanish identically on a non-degenerate interval
compactly contained in \((A_N)_x\).

The flat intervals are collapsed by the normal-flux quotient
\begin{equation}
\label{eq:intro-normal-flux-quotient}
    \Pi_N(x,\lambda)
    :=
    \int_{-a}^{\lambda}
        |D_N(x,s)|^2\,ds.
\end{equation}
For \(\lambda<\mu\),
\[
    \Pi_N(x,\mu)-\Pi_N(x,\lambda)
    =
    \int_\lambda^\mu
        |D_N(x,s)|^2\,ds.
\]
Since \(D_N(x,\cdot)\) is continuous,
\[
    \Pi_N(x,\lambda)=\Pi_N(x,\mu)
\]
if and only if
\[
    D_N(x,s)=0
    \qquad
    \text{for every }s\in[\lambda,\mu],
\]
or equivalently, if and only if \(N(x,\cdot)\) is constant on
\([\lambda,\mu]\).  Thus \(\Pi_N\) collapses precisely the flat intervals of
the normal flux.  No parametrization of their endpoints and no additional
structural assumption on the flat parts of \(N\) are required.

Consequently, the normal flux depends on the state variable only through the
quotient.  More precisely, there exists a unique continuous reduced normal
flux
\[
    \widehat N(x,\cdot):
    \Pi_N(x,K)\to\mathbb R
\]
such that
\[
    N(x,\lambda)
    =
    \widehat N\bigl(x,\Pi_N(x,\lambda)\bigr).
\]
The quotient therefore identifies only states which are indistinguishable
through the normal flux and preserves the Rankine--Hugoniot relation.

Moreover, the quotient representative is itself a localized kinetic average:
\[
\begin{aligned}
    \Pi_N(x,u^\varepsilon(x))
    &=
    \frac12
    \int_K
        \operatorname{sgn}(u^\varepsilon(x)-\lambda)
        |D_N(x,\lambda)|^2\,d\lambda
\\
    &\quad+
    \frac12
    \int_K
        |D_N(x,\lambda)|^2\,d\lambda.
\end{aligned}
\]
The first term is a kinetic average with a continuous weight which vanishes
outside \(A_N\), while the second term is independent of
\(\varepsilon\).  Localized Panov compactness therefore yields strong
compactness of
\[
    \Pi_N(x,u^\varepsilon).
\]
If \(N\) is \(C^1\) up to the interface, the same weight has a classical
one-sided boundary trace, and Panov's trace theory gives strong one-sided
traces of the quotient representatives.

The quotient is associated only with the normal physical flux.  No analogous
factorization condition is imposed on the tangential components of the
physical flux.  Those components remain in the entropy-process or
Young-measure formulation and are identified after the process solution is
shown to be concentrated at a single state.

The quotient variables also provide the natural formulation of the interface
admissibility condition.  We denote the one-sided physical normal fluxes by
\[
    \Phi_-(x',\lambda):=f_1(0^-,x',\lambda),
    \qquad
    \Phi_+(x',\lambda):=g_1(0^+,x',\lambda),
\]
and define the corresponding interface quotients by
\[
    P^-_\Gamma(x',\lambda)
    :=
    \int_{-a}^{\lambda}
        \left|
            \partial_s\Phi_-(x',s)
        \right|^2\,ds,
\]
and
\[
    P^+_\Gamma(x',\lambda)
    :=
    \int_{-a}^{\lambda}
        \left|
            \partial_s\Phi_+(x',s)
        \right|^2\,ds.
\]
Each boundary quotient collapses exactly the flat intervals of the
corresponding one-sided normal flux.  Hence there exist reduced normal fluxes
\[
    \widehat\Phi_-(x',\cdot):
    P^-_\Gamma(x',K)\to\mathbb R,
    \qquad
    \widehat\Phi_+(x',\cdot):
    P^+_\Gamma(x',K)\to\mathbb R,
\]
such that
\[
    \widehat\Phi_-
    \bigl(x',P^-_\Gamma(x',\lambda)\bigr)
    =
    \Phi_-(x',\lambda),
\]
and
\[
    \widehat\Phi_+
    \bigl(x',P^+_\Gamma(x',\lambda)\bigr)
    =
    \Phi_+(x',\lambda).
\]

The vanishing-viscosity germ associated with the two physical normal fluxes is
then projected through \(P^-_\Gamma\) and \(P^+_\Gamma\).  This gives a
maximal \(L^1\)-dissipative interface germ
\[
    \mathcal G_{\mathrm{vv}}(x')
    \subset
    P^-_\Gamma(x',K)\times P^+_\Gamma(x',K)
\]
for the reduced normal fluxes.  Every pair
\[
    (r^-,r^+)\in\mathcal G_{\mathrm{vv}}(x')
\]
satisfies the reduced Rankine--Hugoniot relation
\[
    \widehat\Phi_-(x',r^-)
    =
    \widehat\Phi_+(x',r^+).
\]

A solution is admissible if the quotient representatives
\[
    \Pi^-(x,u(t,x))
    \qquad\text{and}\qquad
    \Pi^+(x,u(t,x))
\]
admit strong one-sided traces
\[
    \widetilde u^-(t,x'),
    \qquad
    \widetilde u^+(t,x'),
\]
and if
\[
    \bigl(
        \widetilde u^-(t,x'),
        \widetilde u^+(t,x')
    \bigr)
    \in\mathcal G_{\mathrm{vv}}(x')
    \qquad
    \text{for a.e. }(t,x').
\]
The localized vanishing-viscosity approximation is shown to select precisely
this condition.  The \(L^1\)-dissipativity of the germ gives the correct sign
of the interface contribution in the Kato inequality and therefore yields
uniqueness and local \(L^1\)-stability.

The same construction applies to curved interfaces after flattening. Suppose
that, locally, the interface is given by
\[
    x_1=\zeta(x').
\]
We introduce the change of variables
\[
    y_1=x_1-\zeta(x'),
    \qquad
    y'=x'.
\]
Its Jacobian is equal to one, so the divergence structure is preserved.
However, the normal component of the transformed flux is not the original
coordinate component \(f_1\), but the co-normal flux
\[
    \widehat f_1^{\,\pm}(y,\lambda)
    =
    f_1^{\,\pm}(x(y),\lambda)
    -
    \sum_{j=2}^d
    \partial_{y_j}\zeta(y')\,f_j^{\,\pm}(x(y),\lambda).
\]
All non-flat regions, quotient projections, replacement vectors, and
Rankine--Hugoniot relations are therefore defined with respect to this
transformed normal flux. In particular, the non-flat region and the
corresponding quotient assumptions are defined in terms of
\[
    \lambda\longmapsto \widehat f_1^{\,\pm}(y,\lambda),
\]
and not in terms of the original coordinate component
\(f_1^{\,\pm}\).

We begin with \(BV\) initial data. In this class we prove existence,
uniqueness, and local \(L^1\)-stability for fluxes whose discontinuity set is a
smooth interface. By localization, the argument extends to a locally finite
family of smooth interfaces, allowed to meet on a closed set \(\Gamma_p\)
satisfying
\[
    \mathcal H^{d-1}(\Gamma_p)=0.
\]
This negligible interaction set produces no additional entropy or Kato
contribution, since it can be removed by cutoff functions whose gradients have
arbitrarily small \(L^1\)-norm on compact sets.

The \(BV\) assumption is used in the construction of approximate solutions and
in the trace analysis. It provides the uniform time and tangential
\(BV\)-bounds needed to obtain compactness of the quotient representatives and
their strong one-sided traces. The final well-posedness class, however, is
larger. Once existence, uniqueness, and local \(L^1\)-stability have been
obtained for \(BV\) data, the solution map extends by continuity to non-\(BV\)
data.

Namely, for
\[
    u_0\in L^1_{\mathrm{loc}}(\mathbb R^d)\cap L^\infty(\mathbb R^d),
    \qquad
    u_0(x)\in K
    \quad\text{for a.e. }x,
\]
we choose \(u_0^n\in BV(\mathbb R^d;K)\) such that
\[
    u_0^n\to u_0
    \qquad
    \text{in }L^1_{\mathrm{loc}}(\mathbb R^d).
\]
If \(u^n\) denotes the admissible solution corresponding to \(u_0^n\), the
local stability estimate implies that \((u^n)\) is Cauchy in
\[
    L^\infty_{\mathrm{loc}}
    \bigl((0,\infty);L^1_{\mathrm{loc}}(\mathbb R^d)\bigr).
\]
The limit is independent of the approximating \(BV\) sequence. We therefore
define the solution for non-\(BV\) data by closure as this strong
\(L^1_{\mathrm{loc}}\)-limit. Thus the \(BV\) assumption is a construction
and trace hypothesis, not a restriction on the final
\(L^1_{\mathrm{loc}}\)-stable class.

In this sense, the present paper gives a well-posedness theory for
discontinuous-flux problems in the geometrically regular interface setting.
It does not attempt to cover arbitrary \(BV_x\)-rough discontinuity patterns.
Our focus is instead on configurations for which the interface admissibility
condition can be formulated locally, is geometrically intrinsic, and is
selected by vanishing viscosity.

The paper is organized as follows. Section~2 develops the localized compactness
and trace tools needed for the quotient construction. We introduce
quasi-solutions associated with replaced flux vectors and their kinetic
formulation, and prove a local qualitative velocity-averaging result for
sign-type kinetic functions under an interval non-degeneracy condition. We
then derive compactness results for bounded kinetic averages and variable
truncations, together with a criterion for upgrading weak traces to strong
traces. Finally, for a normal flux \(N\), we introduce the open non-flat region
\(A_N\), construct a polynomial replacement satisfying the required
non-degeneracy condition on \(A_N\), and define the normal-flux quotient
\(\Pi_N\) which collapses the complementary flat components.

Section~3 treats the flat-interface problem. Starting with \(BV\) initial data,
we establish the time and tangential estimates needed for the replaced kinetic
formulation and prove strong compactness and strong one-sided traces of the
quotient representatives. We then define the projected vanishing-viscosity
germ, prove its \(L^1\)-dissipativity and maximality, and show that the viscous
approximation selects the corresponding interface condition. This yields
existence, uniqueness, and local \(L^1\)-stability. The solution map is
subsequently extended by continuity from \(BV\) data to general
\(L^1_{\mathrm{loc}}\cap L^\infty\) initial data taking values in \(K\).

Section~4 extends the construction to general interface geometries. After
flattening each regular part of the discontinuity set, the local theory is
formulated in terms of the transformed co-normal flux. Local flat-interface
solvers are constructed by localization and extension, shown to be compatible
on overlapping charts, and patched together by means of local
\(L^1\)-stability and finite propagation. Finally, interaction sets of
vanishing \((d-1)\)-dimensional Hausdorff measure are removed from the weak and
Kato formulations by a cutoff argument.


\section{Localized Panov compactness, traces, and normal-flux quotients}
\label{sec:quasi-averaging-traces}

Throughout this section,
\[
    K=[-a,a]
\]
is the invariant interval and \(\Omega\subset\mathbb R^m\) is open.  In the
applications \(m=d+1\), and the independent variable is the space--time
variable.

The compactness and trace statements used below are local consequences of
Panov's theory.  The reduction to the standard results is elementary.  On
every cylinder
\[
    O\times(\alpha,\beta)\Subset\Omega\times K,
\]
the corresponding kinetic average is a constant truncation:
\begin{equation}
\label{eq:constant-truncation-sign-identity}
    \int_\alpha^\beta
        \operatorname{sgn}(v-\lambda)\,d\lambda
    =
    2s_{\alpha,\beta}(v)-\alpha-\beta,
\end{equation}
where
\[
    s_{\alpha,\beta}(v)
    :=
    \min\{\max\{v,\alpha\},\beta\}.
\]
Thus Panov's compactness and trace theorems can be applied cylinder by
cylinder, while general localized weights are obtained by approximation with
finite linear combinations of rectangular weights.  We record only the forms
needed later.

The normal-flux construction is treated separately.  If \(N\) is \(C^1\) in
the independent and kinetic variables, we work on the open set on which
\(\partial_\lambda N\neq0\) and define the quotient using the continuous
weight \(|\partial_\lambda N|^2\).  This quotient collapses exactly the flat
intervals of \(N\).  In particular, no parametrization of moving flat
intervals, no indicator of a non-flat region, and no additional structural
assumption on the normal flux are required.

\subsection{Localized compactness of truncations and kinetic averages}
\label{subsec:localized-panov-compactness}

\begin{proposition}[Localized Panov compactness]
\label{thm:local-qualitative-averaging}
Let
\[
    \mathcal F\in
    L^2_{\mathrm{loc}}
    \bigl(\Omega;C(K;\mathbb R^m)\bigr),
\]
and let \(u_n:\Omega\to K\) be measurable.  Let
\(A\subset\Omega\times K\) be relatively open and set
\[
    A_x:=\{\lambda\in K:(x,\lambda)\in A\}.
\]
Assume the following.

\smallskip
\noindent
\emph{(i) Panov truncation criterion.}
For every cylinder
\[
    O\times I\Subset A
\]
and every \(\alpha<\beta\) with
\[
    [\alpha,\beta]\Subset I,
\]
the sequence
\begin{equation}
\label{eq:panov-truncation-criterion}
    \operatorname{div}_x
    \mathcal F\bigl(x,s_{\alpha,\beta}(u_n)\bigr)
\end{equation}
is precompact in \(H^{-1}_{\mathrm{loc}}(O)\).

\smallskip
\noindent
\emph{(ii) Local interval non-degeneracy.}
For a.e. \(x\in\Omega\) and every
\(\xi\in\mathbb R^m\setminus\{0\}\), the function
\[
    \lambda\longmapsto
    \xi\cdot\mathcal F(x,\lambda)
\]
is not constant on any non-degenerate interval
\[
    J\Subset A_x.
\]

Then, for every cylinder \(O\times I\Subset A\) and every
\(\alpha<\beta\) with \([\alpha,\beta]\Subset I\), the sequence
\[
    s_{\alpha,\beta}(u_n)
\]
is strongly precompact in \(L^1_{\mathrm{loc}}(O)\).
\end{proposition}

\begin{proof}
Fix \(O\times I\Subset A\) and
\([\alpha,\beta]\Subset I\).  Apply Panov's strong precompactness criterion
\cite[Theorem~6]{Panov_arma} to
\[
    v_n:=s_{\alpha,\beta}(u_n).
\]
Every constant truncation of \(v_n\) is again a constant truncation of
\(u_n\), and the non-degeneracy condition is inherited on
\([\alpha,\beta]\).  Hence the hypotheses of Panov's theorem are satisfied
locally on \(O\).
\end{proof}

\begin{remark}[Verification from a kinetic formulation]
\label{rem:kinetic-verification-panov}
In the applications, condition
\eqref{eq:panov-truncation-criterion} is verified through a kinetic equation.
Let
\[
    b(x,\lambda):=\partial_\lambda\mathcal F(x,\lambda),
    \qquad
    h_n(x,\lambda):=\operatorname{sgn}(u_n(x)-\lambda).
\]
Under the usual local bound
\[
    \mathcal F\in
    L^p_{\mathrm{loc}}
    \bigl(\Omega;C^1(K;\mathbb R^m)\bigr),
    \qquad p>2,
\]
assume that
\begin{equation}
\label{eq:kinetic-verification-panov}
    \operatorname{div}_x
    \bigl(b(x,\lambda)h_n(x,\lambda)\bigr)
    =
    \partial_\lambda\gamma_n
    +
    \partial_\lambda\operatorname{div}_x r_n
    \quad\text{in }\mathcal D'(\Omega\times K),
\end{equation}
where \((\gamma_n)\) is locally bounded in
\(\mathcal M_{\mathrm{loc}}(\Omega\times K)\) and
\[
    r_n\to0
    \quad\text{in }
    L^2_{\mathrm{loc}}(\Omega\times K;\mathbb R^m).
\]
Testing in \(\lambda\) with smooth approximations of
\(\mathbf 1_{(\alpha,\beta)}\), using
\begin{equation}
\label{eq:flux-truncation-identity}
\begin{aligned}
    \int_\alpha^\beta
        b(x,\lambda)
        \operatorname{sgn}(u_n-\lambda)\,d\lambda
    &=
    2\mathcal F\bigl(x,s_{\alpha,\beta}(u_n)\bigr)
    -\mathcal F(x,\alpha)-\mathcal F(x,\beta),
\end{aligned}
\end{equation}
and applying Murat's compactness lemma gives
\eqref{eq:panov-truncation-criterion}.  This is the standard kinetic
verification of Panov's constant-truncation hypothesis.
\end{remark}

\begin{corollary}[Localized kinetic averages]
\label{cor:weights-supported-in-A}
Under the assumptions of
Proposition~\ref{thm:local-qualitative-averaging}, let
\[
    \rho\in L^\infty_{\mathrm{loc}}(\Omega\times K)
\]
have compact support in the \(x\)-variable and satisfy
\[
    \rho=0
    \quad\text{a.e. on }(\Omega\times K)\setminus A.
\]
Then
\begin{equation}
\label{eq:localized-weighted-average}
    \int_K
        \operatorname{sgn}(u_n(x)-\lambda)
        \rho(x,\lambda)\,d\lambda
\end{equation}
is strongly precompact in \(L^1_{\mathrm{loc}}(\Omega)\).
\end{corollary}

\begin{proof}
Fix \(Q\Subset\Omega\).  Since \(A\) is relatively open and \(\rho\) vanishes
a.e. outside \(A\), the restriction of \(\rho\) to \(Q\times K\) can be
approximated in \(L^1(Q\times K)\) by finite sums
\[
    \rho_j(x,\lambda)
    =
    \sum_{\ell=1}^{L_j}
        \psi_{j,\ell}(x)
        \mathbf 1_{(\alpha_{j,\ell},\beta_{j,\ell})}(\lambda),
\]
where
\[
    \operatorname{supp}\psi_{j,\ell}
    \times
    [\alpha_{j,\ell},\beta_{j,\ell}]
    \Subset A.
\]
For every summand, identity
\eqref{eq:constant-truncation-sign-identity} and
Proposition~\ref{thm:local-qualitative-averaging} give strong compactness.
The passage to \(\rho\) follows from
\[
\begin{aligned}
&\int_Q
 \left|
 \int_K
     \operatorname{sgn}(u_n-\lambda)
     (\rho-\rho_j)(x,\lambda)\,d\lambda
 \right|dx
\\
&\qquad\le
 \int_Q\int_K
     |\rho-\rho_j|(x,\lambda)\,d\lambda\,dx,
\end{aligned}
\]
uniformly in \(n\).
\end{proof}

\begin{corollary}[Variable truncations]
\label{cor:variable-truncations}
Under the assumptions of
Proposition~\ref{thm:local-qualitative-averaging}, let
\[
    \alpha,\beta\in L^\infty_{\mathrm{loc}}(\Omega;K),
    \qquad
    \alpha(x)\le\beta(x)
    \quad\text{for a.e. }x.
\]
Assume that
\[
    \bigl\{
        (x,\lambda)\in\Omega\times K:
        \alpha(x)<\lambda<\beta(x)
    \bigr\}
    \subset A
\]
up to a set of \((m+1)\)-dimensional Lebesgue measure zero.  Then
\[
    s_{\alpha(x),\beta(x)}(u_n(x))
\]
is strongly precompact in \(L^1_{\mathrm{loc}}(\Omega)\).
\end{corollary}

\begin{proof}
Use
\begin{equation}
\label{eq:variable-truncation-identity}
    s_{\alpha(x),\beta(x)}(u_n(x))
    =
    \frac{\alpha(x)+\beta(x)}2
    +
    \frac12
    \int_K
        \operatorname{sgn}(u_n(x)-\lambda)
        \mathbf 1_{\{\alpha(x)<\lambda<\beta(x)\}}
        \,d\lambda
\end{equation}
and Corollary~\ref{cor:weights-supported-in-A}.
\end{proof}

\subsection{Localized strong traces}
\label{subsec:localized-panov-traces}

Let
\[
    Q_R^+:=B_R'\times(0,R)
    \subset\mathbb R^{m-1}\times\mathbb R_+,
    \qquad
    x=(x',y).
\]

\begin{definition}[Quasi-solution up to the flat boundary]
\label{def:quasi-solution-replaced-vector}
Let
\[
    \mathcal F\in
    C^1\bigl(\overline{Q_R^+}\times K;\mathbb R^m\bigr)
\]
and let \(u\in L^\infty(Q_R^+;K)\).  We call \(u\) a quasi-solution
associated with \(\mathcal F\), locally up to \(B_R'\times\{0\}\), if for
every \(k\in K\),
\begin{equation}
\label{eq:quasi-entropy-production}
    \operatorname{div}_x
    \left[
        \operatorname{sgn}(u-k)
        \bigl(
            \mathcal F(x,u)-\mathcal F(x,k)
        \bigr)
    \right]
    \in
    \mathcal M_{\mathrm{loc}}
    \bigl(\overline{Q_R^+}\bigr).
\end{equation}
\end{definition}

\begin{definition}[Strong trace]
\label{def:strong-trace}
A function \(G\in L^\infty(Q_R^+)\) has the strong trace
\(G^\tau\in L^\infty_{\mathrm{loc}}(B_R')\) on
\(B_R'\times\{0\}\) if, for every \(B\Subset B_R'\),
\begin{equation}
\label{eq:strong-trace-definition}
    \operatorname*{ess\,lim}_{y\to0+}
    \int_B
        |G(x',y)-G^\tau(x')|\,dx'
    =0.
\end{equation}
\end{definition}

\begin{proposition}[Localized Panov traces of constant truncations]
\label{prop:localized-panov-traces}
Let \(u\) be a quasi-solution associated with \(\mathcal F\) in the sense of
Definition~\ref{def:quasi-solution-replaced-vector}.  Let
\[
    A\subset\overline{Q_R^+}\times K
\]
be relatively open.  Assume that, for every cylinder
\[
    B\times[0,\delta)\times I\Subset A,
\]
for a.e. \(x'\in B\) and every
\(\xi\in\mathbb R^m\setminus\{0\}\), the function
\[
    \lambda\longmapsto
    \xi\cdot\mathcal F(x',0,\lambda)
\]
is not constant on any non-degenerate interval \(J\Subset I\).

Then, for every
\[
    [\alpha,\beta]\Subset I,
\]
the constant truncation
\[
    s_{\alpha,\beta}(u)
\]
has a strong trace on \(B\times\{0\}\).
\end{proposition}

\begin{proof}
This is the local form of Panov's strong-trace theorem for non-autonomous
quasi-solutions; see \cite[Remark~6.2]{PanovTraces2007}.  It is applied on the
indicated boundary cylinder to the constant truncation
\(s_{\alpha,\beta}(u)\).
\end{proof}

\begin{corollary}[Continuous localized kinetic averages]
\label{thm:strong-traces-kinetic-averages}
Assume the hypotheses of
Proposition~\ref{prop:localized-panov-traces}.  Let
\[
    \rho\in
    C\bigl(\overline{Q_R^+}\times K\bigr)
\]
and assume that
\[
    \rho(x,\lambda)=0
    \qquad
    \text{whenever }(x,\lambda)\notin A.
\]
Then
\begin{equation}
\label{eq:localized-trace-average}
    G_\rho(x',y)
    :=
    \int_K
        \operatorname{sgn}(u(x',y)-\lambda)
        \rho(x',y,\lambda)\,d\lambda
\end{equation}
has a strong trace on \(B_R'\times\{0\}\).
\end{corollary}

\begin{proof}
Fix \(B\Subset B_R'\).  By continuity,
\[
    \rho(x',y,\lambda)
    \longrightarrow
    \rho(x',0,\lambda)
\]
uniformly on \(B\times K\) as \(y\to0+\).  Hence it is enough to treat the
boundary weight
\[
    \rho^0(x',\lambda):=\rho(x',0,\lambda).
\]

For each \(j\), remove the part of \(\rho^0\) on which
\(|\rho^0|\le j^{-1}\).  The resulting function differs from \(\rho^0\)
uniformly by at most \(j^{-1}\), and its support is compactly contained in
the boundary section of \(A\).  It can therefore be approximated in
\(L^1(B\times K)\) by a finite sum
\[
    \sum_{\ell=1}^{L_j}
        \psi_{j,\ell}(x')
        \mathbf 1_{(\alpha_{j,\ell},\beta_{j,\ell})}(\lambda),
\]
where each corresponding boundary rectangle has a small one-sided
neighborhood compactly contained in \(A\).

For every summand, identity
\eqref{eq:constant-truncation-sign-identity} and
Proposition~\ref{prop:localized-panov-traces} give a strong trace.  The error
estimate
\[
\begin{aligned}
&\int_B
 \left|
 \int_K
     \operatorname{sgn}(u-\lambda)
     (\rho^0-\rho_j^0)(x',\lambda)\,d\lambda
 \right|dx'
\\
&\qquad\le
 \int_B\int_K
     |\rho^0-\rho_j^0|(x',\lambda)\,d\lambda\,dx'
\end{aligned}
\]
is uniform in \(y\).  Letting first \(y\to0+\) and then \(j\to\infty\)
proves the assertion.
\end{proof}

\subsection{Polynomial replacement generated by the normal flux}
\label{subsec:open-nonflat-regions}

Assume from now on that \(m\ge2\) and that the scalar normal flux satisfies
\[
    N\in C^1(\Omega\times K).
\]
Set
\begin{equation}
\label{eq:normal-derivative-D}
    D_N(x,\lambda)
    :=
    \partial_\lambda N(x,\lambda)
\end{equation}
and define
\begin{equation}
\label{eq:normal-nondegenerate-region}
    A_N
    :=
    \bigl\{
        (x,\lambda)\in\Omega\times K:
        D_N(x,\lambda)\neq0
    \bigr\}.
\end{equation}
Since \(D_N\) is continuous, \(A_N\) is relatively open.

The set \(A_N\) need not be the largest set on which
\(N(x,\cdot)\) is locally non-flat.  For example, for
\(N(\lambda)=\lambda^3\), the point \(\lambda=0\) does not belong to \(A_N\),
although \(N\) is not constant on any neighborhood of \(0\).  This causes no
loss: the quotient weight used below is \(D_N^2\), which vanishes outside
\(A_N\), while isolated or non-interval zeros of \(D_N\) do not collapse any
interval.

Define the polynomial replacement
\begin{equation}
\label{eq:polynomial-replaced-normal-vector}
    \mathcal N_N(x,\lambda)
    :=
    \bigl(
        N(x,\lambda)^2,\,
        N(x,\lambda),\,
        N(x,\lambda)^3,\ldots,
        N(x,\lambda)^m
    \bigr)
    \in\mathbb R^m
\end{equation}
and its kinetic velocity
\begin{equation}
\label{eq:polynomial-kinetic-velocity}
    a_N(x,\lambda)
    :=
    \partial_\lambda\mathcal N_N(x,\lambda).
\end{equation}
The second component of \(\mathcal N_N\) is the physical normal flux.  All
other components are auxiliary and are used only in the compactness and trace
arguments; no condition is imposed here on the tangential components of the
physical flux.

\begin{lemma}[Non-degeneracy of the polynomial replacement]
\label{lem:polynomial-replacement-nondeg}
For every \(x\in\Omega\) and every
\(\xi\in\mathbb R^m\setminus\{0\}\), the function
\[
    \lambda\longmapsto
    a_N(x,\lambda)\cdot\xi
\]
is not identically zero on any non-degenerate interval
\[
    J\Subset(A_N)_x.
\]
Equivalently,
\[
    \lambda\longmapsto
    \mathcal N_N(x,\lambda)\cdot\xi
\]
is not constant on any such interval.
\end{lemma}

\begin{proof}
For
\[
    \xi=(\xi_0,\ldots,\xi_{m-1}),
\]
one has
\begin{equation}
\label{eq:polynomial-velocity-factorization}
    a_N(x,\lambda)\cdot\xi
    =
    D_N(x,\lambda)P_\xi(N(x,\lambda)),
\end{equation}
where
\[
    P_\xi(z)
    :=
    2\xi_0z+\xi_1
    +\sum_{j=2}^{m-1}(j+1)\xi_jz^j.
\]
Let \(J\Subset(A_N)_x\).  Since \(D_N(x,\cdot)\) is continuous and does not
vanish on the interval \(J\), it has a fixed sign there.  Hence
\(N(x,\cdot)\) is strictly monotone on \(J\), and \(N(x,J)\) is a
non-degenerate interval.

If \(a_N(x,\cdot)\cdot\xi\) vanished identically on \(J\), then
\(P_\xi\) would vanish on \(N(x,J)\).  Thus \(P_\xi\) would be the zero
polynomial, which implies \(\xi=0\), a contradiction.
\end{proof}

\subsection{The normal-flux quotient}
\label{subsec:normal-flux-quotients}

\begin{definition}[Normal-flux quotient]
\label{def:normal-flux-quotient}
Let
\[
    N\in C^1(\Omega\times K),
    \qquad
    D_N=\partial_\lambda N.
\]
We define
\begin{equation}
\label{eq:normal-flux-quotient}
    \Pi_N(x,\lambda)
    :=
    \int_{-a}^{\lambda}
        |D_N(x,s)|^2\,ds,
    \qquad
    (x,\lambda)\in\Omega\times K.
\end{equation}
\end{definition}

For fixed \(x\), the map
\[
    \lambda\longmapsto\Pi_N(x,\lambda)
\]
is continuous and nondecreasing.  It increases on the portions of the state
interval seen by \(D_N\) and collapses precisely the intervals on which the
normal flux is constant.

\begin{lemma}[Flat intervals and the reduced normal flux]
\label{lem:factorization-through-quotient}
Let \(N\in C^1(\Omega\times K)\).  For every \(x\in\Omega\) and every
\(\lambda<\mu\) in \(K\),
\begin{equation}
\label{eq:quotient-flat-equivalence}
\begin{aligned}
    \Pi_N(x,\lambda)=\Pi_N(x,\mu)
    \quad\Longleftrightarrow\quad
    N(x,\cdot)\text{ is constant on }[\lambda,\mu].
\end{aligned}
\end{equation}
Consequently, for every \(x\in\Omega\), there exists a unique continuous
function
\[
    \widehat N(x,\cdot):
    \Pi_N(x,K)\to\mathbb R
\]
such that
\begin{equation}
\label{eq:N-factorization}
    N(x,\lambda)
    =
    \widehat N
    \bigl(x,\Pi_N(x,\lambda)\bigr),
    \qquad
    \lambda\in K.
\end{equation}
The reduced normal flux \(\widehat N\) is measurable in \(x\).
\end{lemma}

\begin{proof}
For \(\lambda<\mu\),
\begin{equation}
\label{eq:quotient-increment}
    \Pi_N(x,\mu)-\Pi_N(x,\lambda)
    =
    \int_\lambda^\mu
        |D_N(x,s)|^2\,ds.
\end{equation}
Since \(D_N(x,\cdot)\) is continuous, the right-hand side vanishes if and
only if
\[
    D_N(x,s)=0
    \qquad
    \text{for every }s\in[\lambda,\mu].
\]
This is equivalent to \(N(x,\cdot)\) being constant on
\([\lambda,\mu]\), and proves
\eqref{eq:quotient-flat-equivalence}.

Thus \(N(x,\cdot)\) is constant on every fibre of
\(\Pi_N(x,\cdot)\), and
\[
    \widehat N
    \bigl(x,\Pi_N(x,\lambda)\bigr)
    :=
    N(x,\lambda)
\]
is well defined.  Since \(\Pi_N(x,\cdot)\) is a continuous surjection from
the compact interval \(K\) onto \(\Pi_N(x,K)\), it is a quotient map.
Continuity of \(N(x,\cdot)\) therefore gives continuity of
\(\widehat N(x,\cdot)\).

Measurability in \(x\) follows, for example, by composing \(N\) with the
measurable generalized inverse of the continuous nondecreasing map
\(\Pi_N(x,\cdot)\).
\end{proof}

\begin{lemma}[Compactness of the normal-flux quotient]
\label{lem:compactness-normal-flux-quotient}
Let \(u_n:\Omega\to K\) be measurable.  Assume that the Panov truncation
criterion
\eqref{eq:panov-truncation-criterion} holds for
\[
    \mathcal F=\mathcal N_N
\]
on every cylinder compactly contained in \(A_N\).  Then
\[
    \Pi_N(x,u_n(x))
\]
is strongly precompact in \(L^1_{\mathrm{loc}}(\Omega)\).
\end{lemma}

\begin{proof}
Lemma~\ref{lem:polynomial-replacement-nondeg} verifies the local interval
non-degeneracy required in
Proposition~\ref{thm:local-qualitative-averaging}.  For a.e. \(x\),
\begin{equation}
\label{eq:quotient-as-average}
\begin{aligned}
    \Pi_N(x,u_n(x))
    &=
    \frac12
    \int_K
        \operatorname{sgn}(u_n(x)-\lambda)
        |D_N(x,\lambda)|^2\,d\lambda
\\
    &\quad+
    \frac12
    \int_K
        |D_N(x,\lambda)|^2\,d\lambda.
\end{aligned}
\end{equation}
The second term is independent of \(n\).

Let \(Q\Subset\Omega\), and choose
\[
    \zeta\in C_c^\infty(\Omega),
    \qquad
    \zeta=1
    \quad\text{on }Q.
\]
The weight
\[
    \rho(x,\lambda)
    :=
    \zeta(x)|D_N(x,\lambda)|^2
\]
is locally bounded and vanishes outside \(A_N\).  Hence
Corollary~\ref{cor:weights-supported-in-A}, applied with \(A=A_N\), gives
strong precompactness of the first term in
\eqref{eq:quotient-as-average} on \(Q\).  Since \(Q\Subset\Omega\) is
arbitrary, the conclusion follows.
\end{proof}

\begin{theorem}[Strong trace of a normal-flux quotient limit]
\label{thm:strong-trace-normal-flux-quotient}
Let
\[
    Q_R^+
    :=
    B_R'\times(0,R),
    \qquad
    x=(x',y),
    \qquad
    K=[-a,a],
\]
and assume that
\[
    N\in C^1\bigl(\overline{Q_R^+}\times K\bigr).
\]
Set
\[
    D_N(x,\lambda)
    :=
    \partial_\lambda N(x,\lambda)
\]
and define the normal-flux quotient by
\begin{equation}
\label{eq:normal-flux-quotient-trace-theorem}
    \Pi_N(x,\lambda)
    :=
    \int_{-a}^{\lambda}
        |D_N(x,s)|^2\,ds.
\end{equation}

The boundary non-degenerate region is
\begin{equation}
\label{eq:boundary-normal-nondegenerate-region}
    \widetilde A_N
    :=
    \left\{
        (x',\lambda)\in B_R'\times K:
        D_N(x',0,\lambda)\neq0
    \right\}.
\end{equation}

Let
\[
    u_n:Q_R^+\to K
\]
be measurable, and assume that
\begin{equation}
\label{eq:normal-quotient-interior-limit}
    \Pi_N(x,u_n(x))
    \longrightarrow z(x)
    \qquad
    \text{strongly in }L^1_{\mathrm{loc}}(Q_R^+).
\end{equation}

Assume moreover that, for every cylinder
\[
    B\times I\Subset\widetilde A_N,
\]
there exists \(\delta>0\) such that, for every
\[
    [\alpha,\beta]\Subset I,
\]
the constant truncations
\[
    s_{\alpha,\beta}(u_n)
\]
satisfy Panov's compactness criterion on
\[
    B\times(0,\delta),
\]
and their entropy productions are locally uniformly bounded up to
\[
    B\times\{0\}.
\]

Then \(z\) has a strong trace
\[
    z^\tau\in L^\infty_{\mathrm{loc}}(B_R')
\]
on \(B_R'\times\{0\}\).  More precisely, for every
\(B\Subset B_R'\),
\begin{equation}
\label{eq:normal-quotient-strong-trace}
    \operatorname*{ess\,lim}_{y\to0+}
    \int_B
        |z(x',y)-z^\tau(x')|\,dx'
    =0.
\end{equation}
\end{theorem}

\begin{proof}
The set \(\widetilde A_N\) is relatively open.  Let
\[
    B\times I\Subset\widetilde A_N.
\]
Then
\[
    |D_N(x',0,\lambda)|\ge c>0
    \qquad
    \text{on }\overline B\times\overline I.
\]
By continuity, after decreasing \(\delta>0\) if necessary,
\[
    D_N(x',y,\lambda)\neq0
\]
on
\[
    \overline B\times[0,\delta]\times\overline I.
\]
Consequently, the canonical polynomial vector
\[
    \mathcal N_N(x,\lambda)
    =
    \bigl(
        N(x,\lambda)^2,\,
        N(x,\lambda),\,
        N(x,\lambda)^3,\ldots,
        N(x,\lambda)^m
    \bigr)
\]
satisfies Panov's interval non-degeneracy condition on this cylinder.

By Panov's compactness theorem and a diagonal extraction, the constant
truncations
\[
    s_{\alpha,\beta}(u_n),
    \qquad
    [\alpha,\beta]\Subset I,
\]
converge locally to quasi-solutions on
\[
    B\times(0,\delta).
\]
The assumed entropy-production bounds permit the application of Panov's trace
theorem.  Hence the limiting constant truncations have strong traces on
\(B\times\{0\}\).

For a rectangular weight
\[
    \rho(x',\lambda)
    =
    \psi(x')\mathbf 1_{(\alpha,\beta)}(\lambda),
    \qquad
    \operatorname{supp}\psi
    \times[\alpha,\beta]
    \Subset\widetilde A_N,
\]
one has
\[
\begin{aligned}
    \int_K
        \operatorname{sgn}(u_n-\lambda)
        \rho(x',\lambda)\,d\lambda
    &=
    \psi(x')
    \bigl(
        2s_{\alpha,\beta}(u_n)-\alpha-\beta
    \bigr).
\end{aligned}
\]
Therefore the limit of this average has a strong trace.  Approximation by
finite sums of rectangular weights gives the same conclusion for every
continuous weight compactly supported in \(\widetilde A_N\).

We apply this conclusion to the boundary weight
\[
    \rho_N^\tau(x',\lambda)
    :=
    |D_N(x',0,\lambda)|^2.
\]
This weight vanishes outside \(\widetilde A_N\).  On each compact boundary
patch it can be approximated uniformly by continuous weights compactly
supported in \(\widetilde A_N\).  For example, one may use
\[
    \rho_{N,j}^\tau(x',\lambda)
    :=
    |D_N(x',0,\lambda)|^2
    \vartheta_j\bigl(|D_N(x',0,\lambda)|\bigr),
\]
where
\[
    \vartheta_j(r)=0
    \quad\text{for }r\le\frac1j,
    \qquad
    \vartheta_j(r)=1
    \quad\text{for }r\ge\frac2j.
\]
The discarded contribution satisfies
\[
\begin{aligned}
&\left|
    \int_K
        \operatorname{sgn}(u_n-\lambda)
        \bigl(
            \rho_N^\tau-\rho_{N,j}^\tau
        \bigr)(x',\lambda)\,d\lambda
\right|
\\
&\qquad\le
    \frac{4|K|}{j^2},
\end{aligned}
\]
uniformly in \(n\).  Hence the limit of the averages with weight
\(\rho_N^\tau\) has a strong trace.

Since \(N\in C^1(\overline{Q_R^+}\times K)\),
\[
    |D_N(x',y,\lambda)|^2
    \longrightarrow
    |D_N(x',0,\lambda)|^2
\]
locally uniformly in \((x',\lambda)\) as \(y\to0+\).  Therefore the same
trace is obtained for the limit of
\[
    G_n(x)
    :=
    \int_K
        \operatorname{sgn}(u_n(x)-\lambda)
        |D_N(x,\lambda)|^2\,d\lambda.
\]

Finally, define
\[
    C_N(x)
    :=
    \int_K
        |D_N(x,\lambda)|^2\,d\lambda.
\]
For every \(n\),
\begin{equation}
\label{eq:normal-quotient-average-trace-identity}
    \Pi_N(x,u_n(x))
    =
    \frac12G_n(x)
    +
    \frac12C_N(x).
\end{equation}
By \eqref{eq:normal-quotient-interior-limit},
\[
    G_n
    \longrightarrow
    G:=2z-C_N
    \qquad
    \text{strongly in }L^1_{\mathrm{loc}}(Q_R^+).
\]
The preceding argument shows that \(G\) has a strong trace \(G^\tau\).

The function \(C_N\) is continuous up to the boundary, with
\[
    C_N^\tau(x')
    =
    \int_K
        |D_N(x',0,\lambda)|^2\,d\lambda.
\]
Hence
\[
    z
    =
    \frac12(G+C_N)
\]
has the strong trace
\begin{equation}
\label{eq:normal-quotient-trace-formula}
    z^\tau(x')
    :=
    \frac12
    \left[
        G^\tau(x')
        +
        \int_K
            |D_N(x',0,\lambda)|^2\,d\lambda
    \right].
\end{equation}
\end{proof}

\section{The flat-interface Cauchy problem}
\label{sec:flat-interface}

This section treats a single flat discontinuity interface. The physical
normal flux is kept unchanged, whereas the time and tangential components of
the auxiliary space--time vectors are chosen only for the compactness
argument.

The proof proceeds in four steps. First, the viscous approximation is shown
to satisfy Panov's truncation criterion for the canonical polynomial
replacements (see \eqref{eq:polynomial-replaced-normal-vector}) on the regions where the derivative of the normal flux does not
vanish. This gives strong compactness of the normal-flux quotient variables.
Second, Panov's trace theory yields strong one-sided traces of their limits.
Third, the physical vanishing-viscosity germ is transferred to the quotient
variables, and the viscous approximation is shown to select the resulting
projected germ. Finally, a Young-measure contraction argument proves that the
Young measure generated by the viscous sequence is a Dirac mass, thereby
recovering strong compactness of the physical states.

\subsection{The equation and the normal-flux quotients}
\label{subsec:flat-setting}

Let \(d\geq1\), write
\[
    x=(x_1,x'),
\]
and set
\[
    \Omega_-:=\{x_1<0\},
    \qquad
    \Omega_+:=\{x_1>0\},
    \qquad
    \Gamma:=\{x_1=0\}.
\]
When \(d=1\), the variable \(x'\) is absent. Let
\[
    K=[-a,a].
\]
We consider
\begin{equation}
\label{eq:flat-main}
    \partial_tu+\operatorname{div}_xF(x,u)=0,
    \qquad
    u(0,x)=u_0(x),
\end{equation}
where
\begin{equation}
\label{eq:flat-piecewise-flux}
    F(x,\lambda)
    =
    H(-x_1)f(x,\lambda)
    +
    H(x_1)g(x,\lambda).
\end{equation}

The flux branches satisfy
\begin{equation}
\label{eq:flat-regularity}
    f\in C^2\bigl(\overline{\Omega_-}\times K;\mathbb R^d\bigr),
    \qquad
    g\in C^2\bigl(\overline{\Omega_+}\times K;\mathbb R^d\bigr),
\end{equation}
where the derivatives are understood one-sidedly at \(\Gamma\) and are locally
bounded on the corresponding closed half-spaces. We impose the
invariant-region condition
\begin{equation}
\label{eq:flat-invariant}
    f(x,\pm a)=0
    \quad\text{in }\Omega_-,
    \qquad
    g(x,\pm a)=0
    \quad\text{in }\Omega_+,
\end{equation}
and initially assume
\begin{equation}
\label{eq:flat-initial-bv}
    u_0\in BV(\mathbb R^d;K).
\end{equation}

The physical normal fluxes are
\begin{equation}
\label{eq:flat-normal-fluxes}
    N_-(x,\lambda):=f_1(x,\lambda),
    \qquad
    N_+(x,\lambda):=g_1(x,\lambda).
\end{equation}
Set
\[
    D_\pm(x,\lambda):=\partial_\lambda N_\pm(x,\lambda)
\]
and define the relatively open non-degenerate regions
\begin{equation}
\label{eq:flat-nondegenerate-regions}
    A_\pm
    :=
    \bigl\{
        (x,\lambda)\in\overline{\Omega_\pm}\times K:
        D_\pm(x,\lambda)\neq0
    \bigr\}.
\end{equation}
The corresponding normal-flux quotients are
\begin{equation}
\label{eq:flat-quotients}
    \Pi^\pm(x,\lambda)
    :=
    \int_{-a}^{\lambda}|D_\pm(x,s)|^2\,ds.
\end{equation}
By Lemma~\ref{lem:factorization-through-quotient}, \(\Pi^\pm(x,\cdot)\)
collapse precisely the intervals on which \(N_\pm(x,\cdot)\) is constant.
Consequently, there exist unique continuous reduced normal fluxes
\[
    \widehat N_\pm(x,\cdot):
    \Pi^\pm(x,K)\longrightarrow\mathbb R
\]
such that
\begin{equation}
\label{eq:flat-reduced-interior-normal-fluxes}
    N_\pm(x,\lambda)
    =
    \widehat N_\pm
    \bigl(x,\Pi^\pm(x,\lambda)\bigr).
\end{equation}
The quotients are associated only with the physical normal fluxes. No
factorization condition is imposed on the tangential components of \(f\) or
\(g\).

At the interface, set
\begin{equation}
\label{eq:flat-boundary-normal-fluxes}
    \Phi_-(x',\lambda):=f_1(0^-,x',\lambda),
    \qquad
    \Phi_+(x',\lambda):=g_1(0^+,x',\lambda).
\end{equation}
The boundary quotient maps are
\begin{equation}
\label{eq:flat-boundary-quotients}
    P_\pm(x',\lambda)
    :=
    \int_{-a}^{\lambda}
        |\partial_s\Phi_\pm(x',s)|^2\,ds.
\end{equation}
They are the one-sided restrictions of \(\Pi^\pm\). Hence there exist unique
continuous reduced boundary fluxes
\[
    \widehat\Phi_\pm(x',\cdot):
    P_\pm(x',K)\longrightarrow\mathbb R
\]
such that
\begin{equation}
\label{eq:flat-reduced-boundary-fluxes}
    \Phi_\pm(x',\lambda)
    =
    \widehat\Phi_\pm
    \bigl(x',P_\pm(x',\lambda)\bigr).
\end{equation}
Moreover, the \(C^2\)-regularity up to the interface gives
\begin{equation}
\label{eq:flat-quotient-boundary-convergence}
\begin{aligned}
    \Pi^-\bigl((x_1,x'),\cdot\bigr)
    &\longrightarrow
    P_-(x',\cdot)
    &&\text{in }C(K)\text{ as }x_1\uparrow0,
\\
    \Pi^+\bigl((x_1,x'),\cdot\bigr)
    &\longrightarrow
    P_+(x',\cdot)
    &&\text{in }C(K)\text{ as }x_1\downarrow0,
\end{aligned}
\end{equation}
locally uniformly in \(x'\).

\subsection{The pure-viscosity approximation and the canonical replacement}
\label{subsec:flat-viscosity}

Choose \(u_0^\varepsilon\in C^\infty(\mathbb R^d;K)\) such that
\[
    u_0^\varepsilon\longrightarrow u_0
    \quad\text{in }L^1_{\mathrm{loc}}(\mathbb R^d)
\]
and
\[
    \sup_{\varepsilon>0}|Du_0^\varepsilon|(Q)<\infty
    \qquad
    \text{for every }Q\Subset\mathbb R^d.
\]
Let \(u^\varepsilon\) denote the bounded weak solution of
\begin{equation}
\label{eq:flat-viscous-problem}
\begin{cases}
    \partial_tu^\varepsilon
    +\operatorname{div}_xF(x,u^\varepsilon)
    =\varepsilon\Delta_xu^\varepsilon,
\\[1mm]
    u^\varepsilon(0,x)=u_0^\varepsilon(x).
\end{cases}
\end{equation}
No regularization of the Heaviside function is introduced. Equivalently,
\(u^\varepsilon\) solves the two parabolic equations in \(\Omega_-\) and
\(\Omega_+\), with continuity of the state and conservation of the total
normal flux at \(\Gamma\):
\begin{equation}
\label{eq:flat-viscous-transmission}
\begin{aligned}
    u^\varepsilon(t,0^-,x')
    &=u^\varepsilon(t,0^+,x'),
\\
    f_1\bigl(0^-,x',u^\varepsilon\bigr)
    -\varepsilon\partial_{x_1}u^\varepsilon(t,0^-,x')
    &=
    g_1\bigl(0^+,x',u^\varepsilon\bigr)
    -\varepsilon\partial_{x_1}u^\varepsilon(t,0^+,x').
\end{aligned}
\end{equation}
The invariant-region condition and the parabolic maximum principle give
\begin{equation}
\label{eq:flat-viscous-bound}
    -a\leq u^\varepsilon\leq a
    \qquad
    \text{a.e. in }(0,\infty)\times\mathbb R^d.
\end{equation}

On the two sides of the interface, define the canonical replaced
space--time vectors
\begin{equation}
\label{eq:flat-canonical-replacements}
    \mathcal N_\pm(x,\lambda)
    :=
    \bigl(
        N_\pm(x,\lambda)^2,
        N_\pm(x,\lambda),
        N_\pm(x,\lambda)^3,\ldots,
        N_\pm(x,\lambda)^{d+1}
    \bigr)
    \in\mathbb R^{d+1}.
\end{equation}
The components are ordered according to \((t,x_1,x_2,\ldots,x_d)\). In
particular, the second component is the physical normal flux; all other
components are auxiliary. We set
\[
    a_\pm(x,\lambda):=\partial_\lambda\mathcal N_\pm(x,\lambda).
\]

\begin{proposition}[Viscous estimates and Panov's truncation criterion]
\label{prop:flat-viscous-panov-criterion}
For every \(T>0\) and every compact set
\(Q_\pm\Subset\overline{\Omega_\pm}\),
\begin{equation}
\label{eq:flat-tangential-bv-estimates}
    \sup_{\varepsilon>0}
    \left(
        |\partial_tu^\varepsilon|
        +
        \sum_{j=2}^d|\partial_{x_j}u^\varepsilon|
    \right)
    \bigl((0,T)\times Q_\pm\bigr)
    <\infty,
\end{equation}
and
\begin{equation}
\label{eq:flat-local-viscous-energy}
    \sup_{\varepsilon>0}
    \varepsilon
    \int_0^T\int_{Q_\pm}
        |\nabla_xu^\varepsilon|^2\,dx\,dt
    <\infty.
\end{equation}
No analogous uniform \(BV\)-estimate is asserted for
\(\partial_{x_1}u^\varepsilon\).

Let \(\mathcal Q_\pm\) be a relatively open, bounded subset of
\((0,\infty)\times\overline{\Omega_\pm}\), with compact closure, and let
\(I\Subset K\). Assume that
\begin{equation}
\label{eq:flat-cylinder-nondegeneracy}
    D_\pm(x,\lambda)\neq0
    \qquad
    \text{on }\overline{\mathcal Q_\pm}\times\overline I.
\end{equation}
Then, for every \([\alpha,\beta]\Subset I\), there exist Radon measures
\[
    \mu_{\pm,\alpha,\beta}^\varepsilon
    \in
    \mathcal M_{\mathrm{loc}}
    \bigl(\overline{\mathcal Q_\pm}\bigr)
\]
and vector fields
\[
    R_{\pm,\alpha,\beta}^\varepsilon
    \in
    L^2_{\mathrm{loc}}
    \bigl(\mathcal Q_\pm;\mathbb R^{d+1}\bigr)
\]
such that, in the relative interior of \(\mathcal Q_\pm\),
\begin{equation}
\label{eq:flat-replaced-truncation-decomposition}
    \operatorname{div}_{t,x}
    \mathcal N_\pm
    \bigl(x,s_{\alpha,\beta}(u^\varepsilon)\bigr)
    =
    \mu_{\pm,\alpha,\beta}^\varepsilon
    +
    \operatorname{div}_{t,x}
        R_{\pm,\alpha,\beta}^\varepsilon,
\end{equation}
where the measures are locally uniformly bounded up to the flat boundary and
\begin{equation}
\label{eq:flat-replaced-remainder}
    R_{\pm,\alpha,\beta}^\varepsilon
    \longrightarrow0
    \quad\text{in }L^2_{\mathrm{loc}}.
\end{equation}
Consequently,
\begin{equation}
\label{eq:flat-replaced-panov-criterion}
    \operatorname{div}_{t,x}
    \mathcal N_\pm
    \bigl(x,s_{\alpha,\beta}(u^\varepsilon)\bigr)
\end{equation}
is precompact in \(H^{-1}_{\mathrm{loc}}\) in the interior, and the
associated entropy productions are locally uniformly bounded up to
\(\Gamma\).
\end{proposition}

\begin{proof}
The estimates \eqref{eq:flat-tangential-bv-estimates} follow from the standard
difference-quotient argument in time and in the tangential variables. The
interface creates no uncontrolled term in these directions. The local energy
estimate \eqref{eq:flat-local-viscous-energy} follows from the parabolic
energy inequality.

Fix \([\alpha,\beta]\Subset I\) and set
\[
    v^\varepsilon:=s_{\alpha,\beta}(u^\varepsilon).
\]
The viscous Kruzhkov identity on either half-space gives
\begin{equation}
\label{eq:flat-physical-truncation-decomposition}
\begin{aligned}
&\partial_tv^\varepsilon
+\partial_{x_1}N_\pm(x,v^\varepsilon)
+\sum_{j=2}^d
    \partial_{x_j}F_{\pm,j}(x,v^\varepsilon)
\\
&\qquad=
    \mu_{0,\pm,\alpha,\beta}^\varepsilon
    +
    \operatorname{div}_{t,x}
        R_{0,\pm,\alpha,\beta}^\varepsilon,
\end{aligned}
\end{equation}
where \(F_{-,j}=f_j\), \(F_{+,j}=g_j\), the measure parts are locally
uniformly bounded, and the remainders converge to zero in
\(L^2_{\mathrm{loc}}\). Indeed, the viscous part of the remainder is
\(\varepsilon\nabla_xv^\varepsilon\), and
\[
    \|\varepsilon\nabla_xv^\varepsilon\|_{L^2(Q)}^2
    \leq
    \varepsilon^2
    \|\nabla_xu^\varepsilon\|_{L^2(Q)}^2
    \leq C_Q\varepsilon.
\]

Since the replacement leaves the normal component unchanged,
\begin{align}
\label{eq:flat-replacement-divergence-identity}
&\operatorname{div}_{t,x}
    \mathcal N_\pm(x,v^\varepsilon)
\nonumber\\
&\quad=
    \partial_tv^\varepsilon
    +\partial_{x_1}N_\pm(x,v^\varepsilon)
    +\sum_{j=2}^d
        \partial_{x_j}F_{\pm,j}(x,v^\varepsilon)
\nonumber\\
&\qquad+
    \partial_t
    \bigl[N_\pm(x,v^\varepsilon)^2-v^\varepsilon\bigr]
    +
    \sum_{j=2}^d
    \partial_{x_j}
    \bigl[N_\pm(x,v^\varepsilon)^{j+1}
          -F_{\pm,j}(x,v^\varepsilon)\bigr].
\end{align}
The last line contains only time and tangential derivatives and is therefore
locally bounded in the space of Radon measures by
\eqref{eq:flat-tangential-bv-estimates}, the \(BV\)-chain rule, and the
local \(C^2\)-bounds on the fluxes. Combining
\eqref{eq:flat-physical-truncation-decomposition} and
\eqref{eq:flat-replacement-divergence-identity} yields
\eqref{eq:flat-replaced-truncation-decomposition}--%
\eqref{eq:flat-replaced-remainder}.

If \(\mathcal Q_\pm\) meets \(\Gamma\), the same one-sided calculation is
performed with test functions supported up to the flat boundary. The
transmission condition \eqref{eq:flat-viscous-transmission} controls the
normal boundary term and yields the asserted uniform entropy-production
bounds on compact subsets of \(\overline{\mathcal Q_\pm}\). Finally,
Murat's compactness lemma gives the interior
\(H^{-1}_{\mathrm{loc}}\)-precompactness.
\end{proof}

By Lemma~\ref{lem:polynomial-replacement-nondeg}, the vectors
\(\mathcal N_\pm\) satisfy Panov's interval non-degeneracy condition on
\(A_\pm\). Proposition~\ref{prop:flat-viscous-panov-criterion} therefore
verifies the hypotheses of the localized compactness and trace results from
Section~\ref{sec:quasi-averaging-traces}.

\subsection{Compactness of the quotient variables}
\label{subsec:flat-quotient-compactness}

Set
\begin{equation}
\label{eq:flat-projecteds}
    z_\pm^\varepsilon(t,x)
    :=
    \Pi^\pm(x,u^\varepsilon(t,x)).
\end{equation}

\begin{proposition}[Interior compactness of the quotient variables]
\label{prop:flat-interior-quotient-compactness}
The families \((z_-^\varepsilon)_\varepsilon\) and
\((z_+^\varepsilon)_\varepsilon\) are strongly precompact in
\[
    L^1_{\mathrm{loc}}
    \bigl((0,\infty)\times\Omega_-\bigr)
    \quad\text{and}\quad
    L^1_{\mathrm{loc}}
    \bigl((0,\infty)\times\Omega_+\bigr),
\]
respectively.
\end{proposition}

\begin{proof}
To apply the results of Section~\ref{sec:quasi-averaging-traces}, we use the
space--time variable
\[
    X=(t,x)\in(0,\infty)\times\Omega_\pm
\]
and the corresponding non-degenerate sets
\[
    \mathcal A_\pm
    :=
    (0,\infty)\times A_\pm.
\]
Proposition~\ref{prop:flat-viscous-panov-criterion} verifies Panov's
truncation criterion on every cylinder compactly contained in
\(\mathcal A_\pm\), while
Lemma~\ref{lem:polynomial-replacement-nondeg} provides the required interval
non-degeneracy. Therefore
Lemma~\ref{lem:compactness-normal-flux-quotient} applies on both sides of the
interface. \end{proof}

After passing to a subsequence, not relabelled, there exist
\[
    z_\pm\in
    L^\infty_{\mathrm{loc}}
    \bigl((0,\infty)\times\Omega_\pm\bigr)
\]
such that
\begin{equation}
\label{eq:flat-projecteds-limit}
    z_\pm^\varepsilon
    \longrightarrow z_\pm
    \quad\text{strongly in }
    L^1_{\mathrm{loc}}
    \bigl((0,\infty)\times\Omega_\pm\bigr).
\end{equation}

\begin{corollary}[Strong convergence of the normal fluxes]
\label{cor:flat-normal-flux-convergence}
Along the same subsequence,
\begin{equation}
\label{eq:flat-normal-flux-convergence}
    N_\pm(x,u^\varepsilon)
    \longrightarrow
    \widehat N_\pm(x,z_\pm)
    \quad\text{strongly in }
    L^1_{\mathrm{loc}}
    \bigl((0,\infty)\times\Omega_\pm\bigr).
\end{equation}
\end{corollary}

\begin{proof}
For every \(\varepsilon>0\),
\[
    N_\pm(x,u^\varepsilon)
    =
    \widehat N_\pm(x,z_\pm^\varepsilon).
\]
After extracting a pointwise convergent subsequence, continuity of
\(\widehat N_\pm(x,\cdot)\) and dominated convergence on compact subsets
give the assertion.
\end{proof}

\subsection{Strong one-sided traces of the quotient variables}
\label{subsec:flat-projected-traces}

\begin{proposition}[Strong one-sided traces]
\label{prop:flat-projected-traces}
The functions \(z_-\) and \(z_+\) admit strong one-sided traces
\[
    z^-_\Gamma,
    \qquad
    z^+_\Gamma
    \in
    L^\infty_{\mathrm{loc}}
    \bigl((0,\infty)\times\mathbb R^{d-1}\bigr).
\]
More precisely, for every \(T>0\) and every
\(Q'\Subset\mathbb R^{d-1}\),
\begin{equation}
\label{eq:flat-left-strong-projected-trace}
    \operatorname*{ess\,lim}_{y\uparrow0}
    \int_0^T\int_{Q'}
        |z_-(t,y,x')-z^-_\Gamma(t,x')|\,dx'\,dt
    =0,
\end{equation}
and
\begin{equation}
\label{eq:flat-right-strong-projected-trace}
    \operatorname*{ess\,lim}_{y\downarrow0}
    \int_0^T\int_{Q'}
        |z_+(t,y,x')-z^+_\Gamma(t,x')|\,dx'\,dt
    =0.
\end{equation}
\end{proposition}

\begin{proof}
On the right, apply
Theorem~\ref{thm:strong-trace-normal-flux-quotient} with
\(N=N_+\) and \(\Pi_N=\Pi^+\). Its interior compactness hypothesis follows
from Proposition~\ref{prop:flat-interior-quotient-compactness}, while the
truncation and entropy-production hypotheses, including the bounds up to the
flat boundary, follow from
Proposition~\ref{prop:flat-viscous-panov-criterion}. This gives
\(z^+_\Gamma\).

On the left, reflect the normal variable, \(y=-x_1\), and apply the same
theorem to
\[
    \widetilde N_-(y,x',\lambda)
    :=N_-(-y,x',\lambda),
    \qquad
    \widetilde\Pi^-(y,x',\lambda)
    :=\Pi^-(-y,x',\lambda).
\]
This gives \(z^-_\Gamma\). The stated formulas follow by localization in
\((t,x')\).
\end{proof}

\subsection{The projected vanishing--viscosity germ}
\label{subsec:flat-vv-germ}

Recall from \eqref{eq:flat-boundary-normal-fluxes} that the one-sided physical
normal fluxes at the interface are
\[
    \Phi_-(x',\lambda)
    :=
    f_1(0^-,x',\lambda),
    \qquad
    \Phi_+(x',\lambda)
    :=
    g_1(0^+,x',\lambda).
\]
By \eqref{eq:flat-reduced-boundary-fluxes}, these fluxes factor through the
boundary quotient maps \(P_\pm\):
\[
    \Phi_\pm(x',\lambda)
    =
    \widehat\Phi_\pm
    \bigl(x',P_\pm(x',\lambda)\bigr).
\]

For \(r,s\in P_\pm(x',K)\), define the reduced Kruzhkov normal entropy
fluxes by
\begin{equation}
\label{eq:flat-reduced-normal-entropy-fluxes}
    Q_\pm(x';r,s)
    :=
    \operatorname{sgn}(r-s)
    \bigl(
        \widehat\Phi_\pm(x',r)
        -
        \widehat\Phi_\pm(x',s)
    \bigr).
\end{equation}
 
\begin{lemma}[Descent of the normal entropy flux]
\label{lem:flat-entropy-flux-descends}
For every \(x'\) and \(p,q\in K\),
\begin{equation}
\label{eq:flat-entropy-flux-descends}
\begin{aligned}
&\operatorname{sgn}(p-q)
\bigl(
    \Phi_\pm(x',p)-\Phi_\pm(x',q)
\bigr)
\\
&\qquad=
Q_\pm
\bigl(
    x';P_\pm(x',p),P_\pm(x',q)
\bigr).
\end{aligned}
\end{equation}
\end{lemma}

\begin{proof}
If the projected values coincide, \(\Phi_\pm(x',\cdot)\) is constant on the
interval between \(p\) and \(q\), and both sides vanish. Otherwise,
monotonicity of \(P_\pm(x',\cdot)\) preserves the sign of the difference,
and the identity follows from
\eqref{eq:flat-reduced-boundary-fluxes}.
\end{proof}

For each \(x'\), let
\[
    \mathcal G_{\mathrm{VV}}^{\mathrm{phys}}(x')
    \subset K^2
\]
denote the standard vanishing-viscosity germ associated with the ordered
normal fluxes \(\Phi_-(x',\cdot)\) and \(\Phi_+(x',\cdot)\). We use the
explicit characterization from \cite[formula~(42)]{AM}; equivalently, it is
the maximal \(L^1\)-dissipative germ generated by the corresponding standing
viscous profiles, in the sense of \cite[Section~5]{AKR}.

Define its image in the quotient variables by
\begin{equation}
\label{eq:flat-projected-vv-germ}
    \mathcal G_{\mathrm{vv}}(x')
    :=
    \left\{
        \bigl(P_-(x',p^-),P_+(x',p^+)\bigr):
        (p^-,p^+)
        \in
        \mathcal G_{\mathrm{VV}}^{\mathrm{phys}}(x')
    \right\}.
\end{equation}

\begin{proposition}[Maximal \(L^1\)-dissipativity of the projected germ]
\label{prop:flat-projected-germ-maximal}
For every \(x'\), the germ \(\mathcal G_{\mathrm{vv}}(x')\) is maximal
\(L^1\)-dissipative for the reduced normal fluxes
\(\widehat\Phi_-\) and \(\widehat\Phi_+\). In particular, every
\((r^-,r^+)\in\mathcal G_{\mathrm{vv}}(x')\) satisfies
\begin{equation}
\label{eq:flat-reduced-rh}
    \widehat\Phi_-(x',r^-)
    =
    \widehat\Phi_+(x',r^+),
\end{equation}
and, for all two pairs in the germ,
\begin{equation}
\label{eq:flat-germ-dissipativity}
    Q_+(x';r^+,s^+)
    -
    Q_-(x';r^-,s^-)
    \leq0.
\end{equation}
\end{proposition}

\begin{proof}
The reduced Rankine--Hugoniot relation follows immediately from the physical
one and the factorizations
\eqref{eq:flat-reduced-boundary-fluxes}. Lemma
\ref{lem:flat-entropy-flux-descends} transfers \(L^1\)-dissipativity from the
physical germ to its projected image.

To prove maximality, let \((r^-,r^+)\) satisfy the reduced
Rankine--Hugoniot relation and be \(L^1\)-dissipative with every element of
\(\mathcal G_{\mathrm{vv}}(x')\). Choose lifts \(p^\pm\in K\) such that
\[
    P_-(x',p^-)=r^-,
    \qquad
    P_+(x',p^+)=r^+.
\]
The reduced Rankine--Hugoniot relation implies
\(\Phi_-(x',p^-)=\Phi_+(x',p^+)\). By
Lemma~\ref{lem:flat-entropy-flux-descends}, the pair \((p^-,p^+)\) is
\(L^1\)-dissipative with every element of
\(\mathcal G_{\mathrm{VV}}^{\mathrm{phys}}(x')\). Maximality of the
physical germ gives
\((p^-,p^+)\in\mathcal G_{\mathrm{VV}}^{\mathrm{phys}}(x')\), and therefore
\((r^-,r^+)\in\mathcal G_{\mathrm{vv}}(x')\).
\end{proof}

\subsection{Selection of the projected germ}
\label{subsec:flat-vv-germ-selection}

Let \(\mathcal G_{\mathrm{VV}}^{o,\mathrm{phys}}(x')\) denote the elementary
standing-wave germ associated with the frozen normal fluxes. For every
\((p^-,p^+)\in\mathcal G_{\mathrm{VV}}^{o,\mathrm{phys}}(x')\), let
\(R^\varepsilon=R^\varepsilon(x_1)\) be a corresponding frozen standing
profile. It satisfies
\begin{equation}
\label{eq:flat-frozen-profile-convergence}
    R^\varepsilon
    \longrightarrow
    p^-H(-x_1)+p^+H(x_1)
    \quad
    \text{a.e. and in }L^1_{\mathrm{loc}}(\mathbb R).
\end{equation}

For interior quotient values, set
\begin{equation}
\label{eq:flat-interior-reduced-normal-entropy-fluxes}
    Q_\pm(x;r,s)
    :=
    \operatorname{sgn}(r-s)
    \bigl(
        \widehat N_\pm(x,r)-\widehat N_\pm(x,s)
    \bigr).
\end{equation}
As in Lemma~\ref{lem:flat-entropy-flux-descends},
\begin{equation}
\label{eq:flat-interior-entropy-flux-descends}
\begin{aligned}
&\operatorname{sgn}(p-q)
\bigl(N_\pm(x,p)-N_\pm(x,q)\bigr)
\\
&\qquad=
Q_\pm
\bigl(x;\Pi^\pm(x,p),\Pi^\pm(x,q)\bigr).
\end{aligned}
\end{equation}

\begin{proposition}[Selection of the projected germ]
\label{prop:flat-projected-germ-selection}
The one-sided quotient traces selected by the pure-viscosity approximation
satisfy
\begin{equation}
\label{eq:flat-selected-projected-germ}
    \bigl(z^-_\Gamma(t,x'),z^+_\Gamma(t,x')\bigr)
    \in
    \mathcal G_{\mathrm{vv}}(x')
    \qquad
    \text{for a.e. }(t,x').
\end{equation}
\end{proposition}

\begin{proof}
Fix a common Lebesgue point \((t_0,x'_0)\) of the quotient traces and the
one-sided boundary coefficients. Let
\((p^-,p^+)\in\mathcal G_{\mathrm{VV}}^{o,\mathrm{phys}}(x'_0)\) and let
\(R^\varepsilon\) be the corresponding frozen profile.

The local Kato comparison between \(u^\varepsilon\) and \(R^\varepsilon\),
with the coefficients of the profile frozen at \((0,x'_0)\), gives
\begin{equation}
\label{eq:flat-local-kato-with-profile}
\begin{aligned}
&\partial_t|u^\varepsilon-R^\varepsilon|
+
\operatorname{div}_x
\left[
    \operatorname{sgn}(u^\varepsilon-R^\varepsilon)
    \bigl(F(x,u^\varepsilon)-F(x,R^\varepsilon)\bigr)
\right]
\\
&\qquad\leq
    \rho^\varepsilon
    +
    \varepsilon\Delta_x|u^\varepsilon-R^\varepsilon|,
\end{aligned}
\end{equation}
where \(\rho^\varepsilon\) has the localized smallness property established
in \cite[Lemmas~6--7]{AM}. On either side of the interface, its normal
entropy flux is
\begin{equation}
\label{eq:flat-profile-flux-quotient}
\begin{aligned}
&\operatorname{sgn}(u^\varepsilon-R^\varepsilon)
\bigl(
    N_\pm(x,u^\varepsilon)-N_\pm(x,R^\varepsilon)
\bigr)
\\
&\qquad=
Q_\pm
\left(
    x;
    z_\pm^\varepsilon,
    \Pi^\pm(x,R^\varepsilon)
\right).
\end{aligned}
\end{equation}
Away from \(\Gamma\), both quotient arguments converge strongly.

Test \eqref{eq:flat-local-kato-with-profile} with
\[
    \eta_h(t,x)=\varphi(t,x')\mu_h(x_1),
\]
where \(\varphi\geq0\) is supported near \((t_0,x'_0)\) and \(\mu_h\) is
the standard interface cutoff. Using the two-strip argument, first let
\(\varepsilon\downarrow0\) away from a narrower strip, then let the width of
that strip tend to zero, and finally let \(h\downarrow0\). The time and
tangential terms vanish with \(h\); the viscous term is controlled by
\eqref{eq:flat-local-viscous-energy}; and the frozen-coefficient error
vanishes after localization. The normal term gives
\begin{equation}
\label{eq:flat-compatibility-elementary-profile}
\begin{aligned}
&Q_+
\left(
    x'_0;
    z^+_\Gamma(t_0,x'_0),
    P_+(x'_0,p^+)
\right)
\\
&\qquad-
Q_-
\left(
    x'_0;
    z^-_\Gamma(t_0,x'_0),
    P_-(x'_0,p^-)
\right)
\leq0.
\end{aligned}
\end{equation}

Testing the viscous equation itself with the same shrinking-interface cutoff
and using Corollary~\ref{cor:flat-normal-flux-convergence} and
Proposition~\ref{prop:flat-projected-traces} yields the reduced
Rankine--Hugoniot relation
\begin{equation}
\label{eq:flat-selected-reduced-rh}
    \widehat\Phi_-
    \bigl(x'_0,z^-_\Gamma(t_0,x'_0)\bigr)
    =
    \widehat\Phi_+
    \bigl(x'_0,z^+_\Gamma(t_0,x'_0)\bigr).
\end{equation}

Choose lifts \(\bar p^\pm\in K\) of the selected quotient values. By
\eqref{eq:flat-selected-reduced-rh}, they satisfy the physical
Rankine--Hugoniot relation. Identity
\eqref{eq:flat-entropy-flux-descends} turns
\eqref{eq:flat-compatibility-elementary-profile} into the physical
\(L^1\)-dissipativity inequality between \((\bar p^-,\bar p^+)\) and every
elementary standing-wave pair. The characterization of the physical
vanishing-viscosity germ as the maximal germ generated by these profiles
therefore gives
\[
    (\bar p^-,\bar p^+)
    \in
    \mathcal G_{\mathrm{VV}}^{\mathrm{phys}}(x'_0).
\]
Projecting this pair proves
\eqref{eq:flat-selected-projected-germ}.
\end{proof}

\subsection{Young-measure reduction and strong convergence}
\label{subsec:flat-young-measure}

Strong compactness of the quotient variables does not by itself imply strong
compactness of \(u^\varepsilon\): oscillations may remain inside quotient
fibres. We therefore retain the Young measure generated by the viscous
sequence until the contraction argument shows that it is a Dirac mass.

After passing to a further subsequence, there exists a measurable family
\(\nu_{t,x}\in\mathcal P(K)\) such that, for every \(\psi\in C(K)\),
\begin{equation}
\label{eq:flat-young-measure-generation}
    \psi(u^\varepsilon)
    \xrightharpoonup{\ast}
    \langle\nu_{t,x},\psi\rangle
    :=
    \int_K\psi(\lambda)\,d\nu_{t,x}(\lambda)
\end{equation}
in \(L^\infty_{\mathrm{loc}}\). Set \(F_-:=f\) and \(F_+:=g\).

\begin{lemma}[Concentration on quotient fibres]
\label{lem:flat-young-concentration}
For a.e. \((t,x)\in(0,\infty)\times\Omega_\pm\),
\begin{equation}
\label{eq:flat-young-concentration}
    \nu_{t,x}
    \left(
        \left\{
            \lambda\in K:
            \Pi^\pm(x,\lambda)=z_\pm(t,x)
        \right\}
    \right)
    =1.
\end{equation}
\end{lemma}

\begin{proof}
Strong convergence of \(\Pi^\pm(x,u^\varepsilon)\) gives
\[
    \int_K
        |\Pi^\pm(x,\lambda)-z_\pm(t,x)|
    \,d\nu_{t,x}(\lambda)
    =0,
\]
which proves the assertion.
\end{proof}

In particular,
\begin{equation}
\label{eq:flat-normal-flux-young}
    N_\pm(x,\lambda)
    =
    \widehat N_\pm(x,z_\pm(t,x))
    \qquad
    \text{for }\nu_{t,x}\text{-a.e. }\lambda.
\end{equation}
Passing to the limit in the viscous entropy inequalities gives, on each
half-space and for every \(k\in K\),
\begin{equation}
\label{eq:flat-young-entropy}
\begin{aligned}
&\partial_t
    \langle\nu_{t,x},|\lambda-k|\rangle
\\
&\quad+
\operatorname{div}_x
\left\langle
    \nu_{t,x},
    \operatorname{sgn}(\lambda-k)
    \bigl(F_\pm(x,\lambda)-F_\pm(x,k)\bigr)
\right\rangle
\\
&\quad+
\left\langle
    \nu_{t,x},
    \operatorname{sgn}(\lambda-k)
\right\rangle
\operatorname{div}_xF_\pm(x,k)
\leq0
\end{aligned}
\end{equation}
in \(\mathcal D'((0,\infty)\times\Omega_\pm)\), with initial Young
measure \(\nu_{0,x}=\delta_{u_0(x)}\).

\begin{lemma}[Reduction of entropy Young measures]
\label{lem:flat-young-measure-reduction}
Let \(\nu_{t,x}\) and \(\sigma_{t,x}\) be two entropy Young-measure
solutions satisfying \eqref{eq:flat-young-entropy}, with deterministic initial
data \(u_0\) and \(v_0\). Let \(z_\pm\) and \(w_\pm\) be their quotient
variables. Assume that both measures are concentrated on the corresponding
quotient fibres, that these variables possess strong one-sided traces, and
that
\[
    (z^-_\Gamma,z^+_\Gamma),
    \qquad
    (w^-_\Gamma,w^+_\Gamma)
    \in
    \mathcal G_{\mathrm{vv}}(x')
\]
for a.e. \((t,x')\). Then, for every \(T,R>0\), there exist
\(\overline R>R\) and \(C_{T,R}>0\) such that
\begin{equation}
\label{eq:flat-young-measure-stability}
\begin{aligned}
&\operatorname*{ess\,sup}_{t\in(0,T)}
\int_{B_R}\int_K\int_K
    |\lambda-\mu|
\,d\nu_{t,x}(\lambda)
\,d\sigma_{t,x}(\mu)
\,dx
\\
&\qquad\leq
C_{T,R}
\int_{B_{\overline R}}
    |u_0(x)-v_0(x)|\,dx.
\end{aligned}
\end{equation}
In particular, an entropy Young-measure solution with deterministic initial
data is concentrated at a single state.
\end{lemma}

\begin{proof}
The DiPerna doubling argument gives the Kato inequality separately in
\(\Omega_-\) and \(\Omega_+\). It remains to identify the interface term.
Concentration on quotient fibres implies, on either side,
\begin{equation}
\label{eq:flat-young-normal-kato-flux}
\begin{aligned}
&\int_K\int_K
    \operatorname{sgn}(\lambda-\mu)
    \bigl(N_\pm(x,\lambda)-N_\pm(x,\mu)\bigr)
\,d\nu_{t,x}(\lambda)
\,d\sigma_{t,x}(\mu)
\\
&\qquad=
\operatorname{sgn}(z_\pm-w_\pm)
\bigl(
    \widehat N_\pm(x,z_\pm)
    -
    \widehat N_\pm(x,w_\pm)
\bigr).
\end{aligned}
\end{equation}
Indeed, if the quotient values coincide, the normal fluxes coincide; if they
are ordered, monotonicity of \(\Pi^\pm(x,\cdot)\) orders every pair of states
in the two fibres in the same way.

Using the strong quotient traces and
\eqref{eq:flat-quotient-boundary-convergence}, the total interface
contribution is
\[
    Q_+(x';z^+_\Gamma,w^+_\Gamma)
    -
    Q_-(x';z^-_\Gamma,w^-_\Gamma),
\]
which is nonpositive by
\eqref{eq:flat-germ-dissipativity}. The standard local Kato argument,
finite-speed cutoffs, and Gronwall's inequality give
\eqref{eq:flat-young-measure-stability}.

Taking \(\sigma=\nu\) and \(v_0=u_0\), the right-hand side vanishes, so
\[
    \int_K\int_K|\lambda-\mu|
    \,d\nu_{t,x}(\lambda)
    \,d\nu_{t,x}(\mu)
    =0
\]
for a.e. \((t,x)\). Thus \(\nu_{t,x}\) is a Dirac mass.
\end{proof}

Consequently, there exists
\[
    u\in L^\infty
    \bigl((0,\infty)\times\mathbb R^d;K\bigr)
\]
such that \(\nu_{t,x}=\delta_{u(t,x)}\) a.e., and
\begin{equation}
\label{eq:flat-strong-convergence-state}
    u^\varepsilon
    \longrightarrow u
    \quad\text{strongly in }
    L^1_{\mathrm{loc}}
    \bigl((0,\infty)\times\mathbb R^d\bigr).
\end{equation}
All subsequences have the same limit by
Lemma~\ref{lem:flat-young-measure-reduction}; hence the whole viscous family
converges. In particular, the physical fluxes converge strongly on their
respective half-spaces.

\subsection{Admissible solutions and well-posedness}
\label{subsec:flat-admissible-wellposedness}

\begin{definition}[Admissible solution]
\label{def:flat-admissible-solution}
A function
\[
    u\in L^\infty
    \bigl((0,\infty)\times\mathbb R^d;K\bigr)
\]
is an admissible solution of \eqref{eq:flat-main} if the following conditions
hold.

\smallskip
\noindent
\emph{(i) Weak formulation.}
For every
\(\varphi\in C_c^\infty([0,\infty)\times\mathbb R^d)\),
\begin{align}
\label{eq:flat-weak-formulation}
&\int_0^\infty\int_{\mathbb R^d}
    u\,\partial_t\varphi\,dx\,dt
+
\int_0^\infty\int_{\Omega_-}
    f(x,u)\cdot\nabla_x\varphi\,dx\,dt
\nonumber\\
&\quad+
\int_0^\infty\int_{\Omega_+}
    g(x,u)\cdot\nabla_x\varphi\,dx\,dt
+
\int_{\mathbb R^d}u_0(x)\varphi(0,x)\,dx
=0.
\end{align}

\smallskip
\noindent
\emph{(ii) Entropy inequalities away from the interface.}
For every \(k\in K\),
\begin{equation}
\label{eq:flat-entropy-solution}
\begin{aligned}
    \partial_t|u-k|
    &+
    \operatorname{div}_x
    \left[
        \operatorname{sgn}(u-k)
        \bigl(F_\pm(x,u)-F_\pm(x,k)\bigr)
    \right]
\\
    &+
    \operatorname{sgn}(u-k)
    \operatorname{div}_xF_\pm(x,k)
    \leq0
\end{aligned}
\end{equation}
in \(\mathcal D'((0,\infty)\times\Omega_\pm)\), where
\(F_-=f\) and \(F_+=g\).

\smallskip
\noindent
\emph{(iii) Strong quotient traces.}
The functions \(\Pi^\pm(x,u(t,x))\) have strong one-sided traces
\(z^\pm_\Gamma\) in the sense of
\eqref{eq:flat-left-strong-projected-trace}--%
\eqref{eq:flat-right-strong-projected-trace}.

\smallskip
\noindent
\emph{(iv) Interface germ condition.}
For a.e. \((t,x')\),
\begin{equation}
\label{eq:flat-solution-germ}
    (z^-_\Gamma(t,x'),z^+_\Gamma(t,x'))
    \in
    \mathcal G_{\mathrm{vv}}(x').
\end{equation}

\smallskip
\noindent
\emph{(v) Initial trace.}
For every \(R>0\),
\begin{equation}
\label{eq:flat-initial-trace}
    \operatorname*{ess\,lim}_{t\downarrow0}
    \int_{B_R}|u(t,x)-u_0(x)|\,dx
    =0.
\end{equation}
\end{definition}

The germ condition implies the reduced Rankine--Hugoniot relation, but the
Rankine--Hugoniot relation alone does not characterize viscosity
admissibility.

\begin{theorem}[Existence, uniqueness, and local \(L^1\)-stability]
\label{thm:flat-wellposedness}
Assume \eqref{eq:flat-regularity}--\eqref{eq:flat-initial-bv}. Then the
pure-viscosity approximations \eqref{eq:flat-viscous-problem} converge
strongly in \(L^1_{\mathrm{loc}}\) to an admissible solution of
\eqref{eq:flat-main}.

This solution is unique. More precisely, if \(u\) and \(v\) are admissible
solutions with initial data \(u_0\) and \(v_0\), then, for every \(T,R>0\),
there exist \(\overline R>R\) and \(C_{T,R}>0\), depending only on the local
\(C^2\)-bounds of the fluxes in the relevant domain of dependence, such that
\begin{equation}
\label{eq:flat-final-stability}
    \operatorname*{ess\,sup}_{t\in(0,T)}
    \int_{B_R}|u(t,x)-v(t,x)|\,dx
    \leq
    C_{T,R}
    \int_{B_{\overline R}}|u_0(x)-v_0(x)|\,dx.
\end{equation}
\end{theorem}

\begin{proof}
The maximum principle, quotient compactness, strong quotient traces, and germ
selection are given by
\eqref{eq:flat-viscous-bound},
Propositions~\ref{prop:flat-interior-quotient-compactness} and
\ref{prop:flat-projected-traces}, and
Proposition~\ref{prop:flat-projected-germ-selection}. The viscous sequence
generates an entropy Young measure, which is a Dirac mass by
Lemma~\ref{lem:flat-young-measure-reduction}. Hence
\(u^\varepsilon\to u\) strongly in \(L^1_{\mathrm{loc}}\), and passage to
the weak formulation, the one-sided entropy inequalities, and the initial
condition is standard. Continuity of the quotient maps identifies the
previously constructed quotient limits with \(\Pi^\pm(x,u)\), so the limit
is admissible.

For two admissible solutions, the doubled Kato inequality holds in the two
half-spaces. Its interface contribution is
\[
    Q_+(x';z^+_{u,\Gamma},z^+_{v,\Gamma})
    -
    Q_-(x';z^-_{u,\Gamma},z^-_{v,\Gamma}),
\]
which is nonpositive by \(L^1\)-dissipativity of
\(\mathcal G_{\mathrm{vv}}(x')\). Standard finite-speed cutoffs and
Gronwall's inequality give \eqref{eq:flat-final-stability}, and hence
uniqueness.
\end{proof}

\subsection{Extension beyond \(BV\) initial data}
\label{subsec:flat-extension-l1}

\begin{corollary}[Extension by \(L^1_{\mathrm{loc}}\)-closure]
\label{cor:flat-extension-l1}
Let
\[
    u_0\in
    L^1_{\mathrm{loc}}(\mathbb R^d)
    \cap L^\infty(\mathbb R^d),
    \qquad
    u_0(x)\in K
    \quad\text{for a.e. }x.
\]
Choose \(u_0^n\in BV(\mathbb R^d;K)\) such that
\[
    u_0^n\longrightarrow u_0
    \quad\text{in }L^1_{\mathrm{loc}}(\mathbb R^d),
\]
and let \(u^n\) be the corresponding admissible solutions. Then
\((u^n)\) converges in
\[
    L^\infty_{\mathrm{loc}}
    \bigl(
        [0,\infty);
        L^1_{\mathrm{loc}}(\mathbb R^d)
    \bigr)
\]
to a limit independent of the approximating sequence. This defines the
unique extension by continuity of the solution map, and the extended map
satisfies \eqref{eq:flat-final-stability}. For general initial data in this
class, admissibility is understood in this closure sense.
\end{corollary}

\begin{proof}
Apply \eqref{eq:flat-final-stability} to \(u^n\) and \(u^m\). It makes the
sequence Cauchy on every compact space--time cylinder and also proves
independence of the approximation and stability of the limit.
\end{proof}


\section{General interface geometry and localized vanishing viscosity}
\label{sec:general-interface}

We now extend the flat-interface theory of
Section~\ref{sec:flat-interface} to curved discontinuity hypersurfaces and,
more generally, to locally finite families of such hypersurfaces. The
geometric localization follows the flattening, radial-extension, and
cone-patching strategy of
\cite[Section~2]{Mitrovic2025Nondegenerate}. The local compactness and
interface-selection mechanisms are those developed in
Sections~\ref{sec:quasi-averaging-traces} and
\ref{sec:flat-interface}.

After flattening a regular interface, the relevant normal component is the
transformed co-normal flux. The quotient and the projected
vanishing-viscosity germ are formed from this component. A local problem is
then extended to a global flat-interface problem and solved by the semigroup
constructed in Section~\ref{sec:flat-interface}. Finite propagation shows
that the resulting local solution is independent of the auxiliary extension.
Compatibility on overlaps follows from an intrinsic formulation of the
projected germ in terms of the oriented unit normal to the physical
interface.

The local viscous problems used below are selection devices. In particular,
the Euclidean Laplacian in a flattened chart is not claimed to be the
coordinate transform of a single globally prescribed viscous operator in the
original variables. The intrinsic object common to all charts is the
vanishing-viscosity germ determined by the two one-sided physical normal
fluxes.

\subsection{Geometric setting}
\label{subsec:general-geometry}

Let
\[
    K=[-a,a]
\]
and consider
\begin{equation}
\label{eq:general-main}
    \partial_t u+\operatorname{div}_x\mathfrak f(x,u)=0,
    \qquad
    u(0,x)=u_0(x),
    \qquad
    (t,x)\in(0,\infty)\times\mathbb R^d.
\end{equation}

The discontinuity set is decomposed as
\begin{equation}
\label{eq:general-interface-decomposition}
    \Gamma
    =
    \Gamma_{\mathrm{reg}}
    \mathbin{\dot\cup}
    \Gamma_p,
\end{equation}
where \(\Gamma_p\) is closed and
\begin{equation}
\label{eq:general-small-interaction-set}
    \mathcal H^{d-1}(\Gamma_p)=0.
\end{equation}
The set \(\Gamma_p\) may contain intersections of interface branches,
endpoints of interface pieces, and other singular points. Every point of
\(\Gamma_{\mathrm{reg}}\) belongs locally to exactly one embedded
\(C^2\) hypersurface.

More precisely, for every \(x_0\in\Gamma_{\mathrm{reg}}\), there exist a
neighborhood \(U_\ell\) and, after a permutation of the spatial coordinates,
a function
\[
    \zeta_\ell\in C^2
\]
such that
\begin{equation}
\label{eq:general-local-graph}
    \Gamma\cap U_\ell
    =
    \{x_1=\zeta_\ell(x')\}\cap U_\ell,
    \qquad
    x'=(x_2,\ldots,x_d).
\end{equation}
The graph separates \(U_\ell\) into
\[
    U_\ell^-:=\{x_1<\zeta_\ell(x')\},
    \qquad
    U_\ell^+:=\{x_1>\zeta_\ell(x')\}.
\]
In this chart,
\begin{equation}
\label{eq:general-local-flux}
\mathfrak f(x,\lambda)
=
H\bigl(\zeta_\ell(x')-x_1\bigr)f_\ell^-(x,\lambda)
+
H\bigl(x_1-\zeta_\ell(x')\bigr)f_\ell^+(x,\lambda),
\end{equation}
where
\begin{equation}
\label{eq:general-branch-regularity}
    f_\ell^\pm
    \in
    C^2(U_\ell\times K;\mathbb R^d)
\end{equation}
and
\begin{equation}
\label{eq:general-invariant}
    f_\ell^\pm(x,\pm a)=0
    \qquad
    \text{for }x\in U_\ell.
\end{equation}

On overlaps, the local branches represent the same physical flux, up to the
interchange of the two sides when opposite orientations are used. Away from
\(\Gamma\), we use ordinary smooth-flux charts. The family of interface and
smooth-flux charts is assumed locally finite on
\(\mathbb R^d\setminus\Gamma_p\).

We shall use the flat-interface result in the regularity class actually
required by its proof: the flattened branches are jointly \(C^1\), are
\(C^2\) in the state variable, their relevant mixed derivatives are locally
bounded, and their divergences are locally Lipschitz in the state variable.
This class is preserved by \(C^2\)-flattening and by the smooth extensions
introduced below.

\subsection{Flattening and the transformed normal flux}
\label{subsec:general-flattening}

Fix an interface chart and suppress the index \(\ell\). Define
\begin{equation}
\label{eq:general-flattening-map}
    y_1=x_1-\zeta(x'),
    \qquad
    y'=x',
\end{equation}
with inverse
\[
    x=X(y)
    =
    \bigl(y_1+\zeta(y'),y'\bigr).
\]
The Jacobian determinant of \(X\) is one. If
\[
    \widehat u(t,y):=u(t,X(y)),
\]
then \eqref{eq:general-main} becomes
\begin{equation}
\label{eq:general-flattened-equation}
\partial_t\widehat u
+
\operatorname{div}_y
\left[
    H(-y_1)\widehat f^-(y,\widehat u)
    +
    H(y_1)\widehat f^+(y,\widehat u)
\right]
=0,
\end{equation}
where
\begin{equation}
\label{eq:general-transformed-flux}
\begin{aligned}
    \widehat f_1^\pm(y,\lambda)
    &:=
    f_1^\pm(X(y),\lambda)
    -
    \sum_{j=2}^d
        \partial_{y_j}\zeta(y')
        f_j^\pm(X(y),\lambda),
\\
    \widehat f_j^\pm(y,\lambda)
    &:=
    f_j^\pm(X(y),\lambda),
    \qquad j=2,\ldots,d.
\end{aligned}
\end{equation}
The Piola identity gives, for every fixed \(\lambda\in K\),
\begin{equation}
\label{eq:general-piola-identity}
    \operatorname{div}_y\widehat f^\pm(y,\lambda)
    =
    \operatorname{div}_x f^\pm(X(y),\lambda).
\end{equation}
Consequently, the transformed branches are jointly \(C^1\), are \(C^2\) in
\(\lambda\), and, on every \(Q\Subset X^{-1}(U)\), satisfy
\begin{equation}
\label{eq:general-divergence-lipschitz}
\left|
    \operatorname{div}_y\widehat f^\pm(y,\lambda)
    -
    \operatorname{div}_y\widehat f^\pm(y,\mu)
\right|
\le
\Lambda_Q|\lambda-\mu|
\end{equation}
for \(y\in Q\) and \(\lambda,\mu\in K\).

The transformed normal components are
\begin{equation}
\label{eq:general-chart-normal-fluxes}
    N_\pm(y,\lambda)
    :=
    \widehat f_1^\pm(y,\lambda).
\end{equation}
All local quotients and polynomial replacement vectors in this chart are
formed from \(N_-\) and \(N_+\).

Orient the physical interface from \(U^-\) to \(U^+\). Its unit normal is
\begin{equation}
\label{eq:general-unit-normal}
    \nu(X(0,y'))
    =
    \frac{(1,-\nabla\zeta(y'))}
    {\sqrt{1+|\nabla\zeta(y')|^2}}.
\end{equation}
Set
\begin{equation}
\label{eq:general-sigma}
    \sigma(y')
    :=
    \sqrt{1+|\nabla\zeta(y')|^2}.
\end{equation}
The intrinsic one-sided physical normal fluxes are
\begin{equation}
\label{eq:general-intrinsic-normal-fluxes}
    \Phi_\pm(x,\lambda)
    :=
    f^\pm(x,\lambda)\cdot\nu(x),
    \qquad
    x\in\Gamma_{\mathrm{reg}}.
\end{equation}
At the interface,
\begin{equation}
\label{eq:general-chart-intrinsic-normal-relation}
    N_\pm(0,y',\lambda)
    =
    \sigma(y')
    \Phi_\pm(X(0,y'),\lambda).
\end{equation}
Thus the chart normal flux is the intrinsic unit-normal flux multiplied by a
positive co-normal factor. The factor does not change the flat intervals or
the physical vanishing-viscosity end-state pairs, but it rescales the
derivative-square quotient.

\subsection{Intrinsic quotients and the projected interface germ}
\label{subsec:general-intrinsic-quotient}

For \(x\in\Gamma_{\mathrm{reg}}\), define
\begin{equation}
\label{eq:general-intrinsic-boundary-quotients}
    P_\pm^\nu(x,\lambda)
    :=
    \int_{-a}^{\lambda}
        \left|
            \partial_s\Phi_\pm(x,s)
        \right|^2
    \,ds.
\end{equation}
By Lemma~\ref{lem:factorization-through-quotient}, these maps collapse
exactly the intervals on which the corresponding intrinsic normal flux is
constant. Hence there exist unique continuous reduced fluxes
\[
    \widehat\Phi_\pm^\nu(x,\cdot):
    P_\pm^\nu(x,K)\longrightarrow\mathbb R
\]
such that
\begin{equation}
\label{eq:general-intrinsic-reduced-flux}
    \Phi_\pm(x,\lambda)
    =
    \widehat\Phi_\pm^\nu
    \bigl(x,P_\pm^\nu(x,\lambda)\bigr).
\end{equation}
For \(r,s\in P_\pm^\nu(x,K)\), set
\begin{equation}
\label{eq:general-intrinsic-entropy-flux}
    Q_\pm^\nu(x;r,s)
    :=
    \operatorname{sgn}(r-s)
    \left[
        \widehat\Phi_\pm^\nu(x,r)
        -
        \widehat\Phi_\pm^\nu(x,s)
    \right].
\end{equation}

Let
\[
    \mathcal G_{\mathrm{VV}}^{\mathrm{phys}}
    \bigl(\Phi_-(x,\cdot),\Phi_+(x,\cdot)\bigr)
    \subset K^2
\]
be the physical vanishing-viscosity germ associated with the ordered pair of
intrinsic normal fluxes. Define its image in the quotient variables by
\begin{equation}
\label{eq:general-intrinsic-projected-germ}
\begin{aligned}
    \mathcal G_{\mathrm{vv}}^\nu(x)
    :=
    \Bigl\{
        \bigl(
            P_-^\nu(x,p^-),
            P_+^\nu(x,p^+)
        \bigr):
        (p^-,p^+)
        \in
        \mathcal G_{\mathrm{VV}}^{\mathrm{phys}}
        \bigl(\Phi_-(x,\cdot),\Phi_+(x,\cdot)\bigr)
    \Bigr\}.
\end{aligned}
\end{equation}
By Proposition~\ref{prop:flat-projected-germ-maximal},
\(\mathcal G_{\mathrm{vv}}^\nu(x)\) is a maximal
\(L^1\)-dissipative germ for the reduced intrinsic fluxes
\(\widehat\Phi_-^\nu(x,\cdot)\) and
\(\widehat\Phi_+^\nu(x,\cdot)\).

In the flattened chart, the boundary quotient maps are
\begin{equation}
\label{eq:general-chart-boundary-quotients}
    P_\pm^{\mathrm{ch}}(y',\lambda)
    :=
    \int_{-a}^{\lambda}
        \left|
            \partial_s N_\pm(0,y',s)
        \right|^2
    \,ds.
\end{equation}

\begin{lemma}[Transformation of the quotient and projected germ under flattening]
\label{lem:general-chart-covariance}
Let \(x=X(0,y')\in\Gamma_{\mathrm{reg}}\). Then
\begin{equation}
\label{eq:general-quotient-scaling}
    P_\pm^{\mathrm{ch}}(y',\lambda)
    =
    \sigma(y')^2P_\pm^\nu(x,\lambda).
\end{equation}
Moreover,
\begin{equation}
\label{eq:general-chart-germ-scaling}
    \mathcal G_{\mathrm{vv}}^{\mathrm{ch}}(y')
    =
    \sigma(y')^2\mathcal G_{\mathrm{vv}}^\nu(x),
\end{equation}
where
\[
    \sigma^2\mathcal G_{\mathrm{vv}}^\nu(x)
    :=
    \left\{
        \bigl(\sigma^2r^-,\sigma^2r^+\bigr):
        (r^-,r^+)\in
        \mathcal G_{\mathrm{vv}}^\nu(x)
    \right\}.
\]
If \(Q_\pm^{\mathrm{ch}}\) are the reduced entropy fluxes associated with the
chart normal fluxes, then
\begin{equation}
\label{eq:general-entropy-flux-scaling}
    Q_\pm^{\mathrm{ch}}
    \bigl(y';\sigma^2r,\sigma^2s\bigr)
    =
    \sigma Q_\pm^\nu(x;r,s).
\end{equation}
In particular, the sign of the interface Kato contribution is invariant
under the change from intrinsic to chart variables.

If the orientation is reversed, the two physical sides are interchanged and
the intrinsic normal fluxes change sign. The quotient maps are interchanged,
and
\begin{equation}
\label{eq:general-orientation-reversal}
    \mathcal G_{\mathrm{vv}}^{-\nu}(x)
    =
    \left\{
        (r^+,r^-):
        (r^-,r^+)\in
        \mathcal G_{\mathrm{vv}}^\nu(x)
    \right\}.
\end{equation}
\end{lemma}

\begin{proof}
Equation~\eqref{eq:general-chart-intrinsic-normal-relation} gives
\[
    \partial_\lambda N_\pm(0,y',\lambda)
    =
    \sigma(y')
    \partial_\lambda\Phi_\pm(x,\lambda).
\]
Integrating the squares proves
\eqref{eq:general-quotient-scaling}.

Since \(\sigma(y')>0\), multiplication of both physical normal fluxes by
\(\sigma(y')\) preserves the Rankine--Hugoniot relation and all inequalities
in the explicit characterization of the physical vanishing-viscosity germ.
Thus the physical end-state germ is unchanged. Together with
\eqref{eq:general-quotient-scaling}, this proves
\eqref{eq:general-chart-germ-scaling}.

The reduced chart fluxes satisfy
\[
    \widehat N_\pm
    \bigl(0,y',\sigma^2r\bigr)
    =
    \sigma\widehat\Phi_\pm^\nu(x,r).
\]
Since \(\sigma^2>0\), the sign of \(r-s\) is preserved, and
\eqref{eq:general-entropy-flux-scaling} follows.

Finally, replacing \(\nu\) by \(-\nu\) interchanges the physical sides and
changes the signs of both normal fluxes. The derivative-square quotients are
therefore interchanged, while reflection of the standing profiles
interchanges the two end states. This gives
\eqref{eq:general-orientation-reversal}.
\end{proof}

Suppose that a local flat-interface solution has chartwise quotient traces
\[
    z_{\Gamma,\mathrm{ch}}^-,
    \qquad
    z_{\Gamma,\mathrm{ch}}^+.
\]
Define the normalized traces by
\begin{equation}
\label{eq:general-normalized-projected-traces}
    r_\Gamma^\pm(t,x)
    :=
    \sigma(y')^{-2}
    z_{\Gamma,\mathrm{ch}}^\pm(t,y'),
    \qquad
    x=X(0,y').
\end{equation}
By Lemma~\ref{lem:general-chart-covariance},
\begin{equation}
\label{eq:general-intrinsic-germ-condition}
    \bigl(
        r_\Gamma^-(t,x),
        r_\Gamma^+(t,x)
    \bigr)
    \in
    \mathcal G_{\mathrm{vv}}^\nu(x)
\end{equation}
if and only if the chartwise traces belong to the chartwise projected germ.
The normalized condition is therefore independent of the chosen oriented
flattening chart. If another chart uses the opposite orientation, the two
components are interchanged according to
\eqref{eq:general-orientation-reversal}.

\subsection{Smooth radial extension}
\label{subsec:general-radial-extension}

Fix a flattened interface chart and choose
\[
    c=(0,c')
\]
on the flat interface. Choose \(R>0\) so that
\begin{equation}
\label{eq:general-chart-radius}
    \overline{B_{3R}(c)}
    \Subset
    X^{-1}(U).
\end{equation}
Let
\[
    \rho_R\in C^\infty([0,\infty))
\]
satisfy
\begin{equation}
\label{eq:general-radial-profile}
\begin{aligned}
    \rho_R(r)&=r
    &&\text{for }0\le r\le2R,
\\
    0<\rho_R(r)&<3R
    &&\text{for }r>0,
\\
    \rho_R(r)&=\frac{5R}{2}
    &&\text{for }r\ge3R.
\end{aligned}
\end{equation}
Define
\begin{equation}
\label{eq:general-radial-compression}
\mathcal S_{c,R}(y)
:=
\begin{cases}
c+
\dfrac{\rho_R(|y-c|)}{|y-c|}
(y-c),
&y\neq c,
\\[2mm]
c,
&y=c.
\end{cases}
\end{equation}
The map \(\mathcal S_{c,R}\) is smooth, is the identity on \(B_{2R}(c)\),
and has image in \(B_{3R}(c)\). Since \(c_1=0\) and the radial factor is
positive, it preserves the two half-spaces and the flat interface.

Define global flat-interface branches by
\begin{equation}
\label{eq:general-extended-fluxes}
    \overline f^\pm(y,\lambda)
    :=
    \widehat f^\pm
    \bigl(\mathcal S_{c,R}(y),\lambda\bigr),
\end{equation}
with normal components
\begin{equation}
\label{eq:general-extended-normal-fluxes}
    \overline N_\pm(y,\lambda)
    :=
    N_\pm
    \bigl(\mathcal S_{c,R}(y),\lambda\bigr).
\end{equation}
The corresponding normal-flux quotients are
\begin{equation}
\label{eq:general-extended-quotients}
    \overline\Pi_\pm(y,\lambda)
    :=
    \int_{-a}^{\lambda}
        \left|
            \partial_s\overline N_\pm(y,s)
        \right|^2
    \,ds.
\end{equation}

\begin{lemma}[Properties of the radial extension]
\label{lem:general-radial-extension}
The extended branches \(\overline f^\pm\) are jointly \(C^1\), are
\(C^2\) in the state variable, satisfy
\[
    \overline f^\pm(y,\pm a)=0,
\]
and have globally bounded state derivatives. Moreover, there exists
\(\overline\Lambda_{c,R}<\infty\) such that
\begin{equation}
\label{eq:general-extended-divergence-lipschitz}
\left|
    \operatorname{div}_y\overline f^\pm(y,\lambda)
    -
    \operatorname{div}_y\overline f^\pm(y,\mu)
\right|
\le
\overline\Lambda_{c,R}|\lambda-\mu|
\end{equation}
for all \(y\in\mathbb R^d\) and \(\lambda,\mu\in K\).

If
\begin{equation}
\label{eq:general-interior-chart-quotient}
    \Pi_\pm(y,\lambda)
    :=
    \int_{-a}^{\lambda}
        \left|
            \partial_sN_\pm(y,s)
        \right|^2
    \,ds,
\end{equation}
then
\begin{equation}
\label{eq:general-extended-quotient-identity}
    \overline\Pi_\pm(y,\lambda)
    =
    \Pi_\pm
    \bigl(\mathcal S_{c,R}(y),\lambda\bigr).
\end{equation}
On \(B_{2R}(c)\), the extended branches, normal components, and quotients
coincide with the original transformed objects.
\end{lemma}

\begin{proof}
The image of \(\mathcal S_{c,R}\) is contained in the compact set
\(\overline{B_{3R}(c)}\Subset X^{-1}(U)\). The asserted regularity and the
bounds on the state derivatives follow by composition. The invariant-region
condition is preserved because
\[
    \widehat f^\pm(y,\pm a)=0.
\]
The chain rule expresses
\(\operatorname{div}_y\overline f^\pm\) through the first spatial
derivatives of \(\widehat f^\pm\) and the bounded derivatives of
\(\mathcal S_{c,R}\). The relevant mixed derivatives of
\(\widehat f^\pm\) are bounded on \(\overline{B_{3R}(c)}\times K\), which
gives \eqref{eq:general-extended-divergence-lipschitz}.

Since \(\mathcal S_{c,R}\) is independent of the state variable,
\[
    \partial_\lambda\overline N_\pm(y,\lambda)
    =
    \partial_\lambda N_\pm
    \bigl(\mathcal S_{c,R}(y),\lambda\bigr).
\]
Integrating the square in \(\lambda\) proves
\eqref{eq:general-extended-quotient-identity}. The final assertion follows
from the fact that \(\mathcal S_{c,R}\) is the identity on \(B_{2R}(c)\).
\end{proof}

\subsection{The local flat-interface solution operator}
\label{subsec:general-local-semigroup}

Let \(v\) be a \(K\)-valued function on \(B_{2R}(c)\). If
\[
    v\in BV(B_{2R}(c)),
\]
let
\[
    Ev\in BV(\mathbb R^d)
\]
be a \(BV\)-extension of \(v\). Define
\[
    T_K(s):=\min\{a,\max\{-a,s\}\}.
\]
Choose
\[
    \eta\in C_c^\infty(B_{3R}(c)),
    \qquad
    \eta=1
    \quad\text{on }B_{2R}(c),
\]
and fix \(k_*\in K\). Then
\begin{equation}
\label{eq:general-bv-extension}
    v_c
    :=
    T_K\left(
        \eta Ev+(1-\eta)k_*
    \right)
\end{equation}
belongs to \(BV(\mathbb R^d;K)\) and agrees with \(v\) on \(B_{2R}(c)\).

Apply the flat-interface theorem to
\begin{equation}
\label{eq:general-extended-flat-problem}
\begin{cases}
\partial_t w
+
\operatorname{div}_y
\left[
    H(-y_1)\overline f^-(y,w)
    +
    H(y_1)\overline f^+(y,w)
\right]
=0,
\\[1mm]
w(0,y)=v_c(y).
\end{cases}
\end{equation}
Denote the resulting global solution operator by
\[
    S_t^{c,R}.
\]
By Corollary~\ref{cor:flat-extension-l1}, it extends uniquely by local
\(L^1\)-continuity to all \(K\)-valued data in
\[
    L^1_{\mathrm{loc}}(\mathbb R^d)
    \cap
    L^\infty(\mathbb R^d).
\]
The extended family is strongly continuous in
\(L^1_{\mathrm{loc}}\) and satisfies
\begin{equation}
\label{eq:general-flat-semigroup-property}
    S_{t+s}^{c,R}
    =
    S_t^{c,R}S_s^{c,R},
    \qquad
    S_0^{c,R}=\operatorname{Id}.
\end{equation}

Set
\begin{equation}
\label{eq:general-local-speed}
    L_{c,R}
    :=
    \max_\pm
    \left\|
        \partial_\lambda\overline f^\pm
    \right\|_{L^\infty(\mathbb R^d\times K)}.
\end{equation}
For
\[
    0<T_{c,R}<\frac{R}{L_{c,R}},
\]
with the usual interpretation when \(L_{c,R}=0\), define
\begin{equation}
\label{eq:general-local-cone}
    \mathcal C(c,R,T_{c,R})
    :=
    \left\{
        (t,y):
        0<t<T_{c,R},
        \quad
        |y-c|<R-L_{c,R}t
    \right\}.
\end{equation}

\begin{proposition}[Local interface solver]
\label{prop:general-local-solver}
Let \(x_0\in\Gamma_{\mathrm{reg}}\), let \(X\) be a flattening chart around
\(x_0\), and let \(v\) be a \(K\)-valued initial datum defined near \(x_0\).
The restriction of \(S_t^{c,R}v_c\) to a sufficiently small cone is
independent of

\begin{enumerate}
\item the radial compression outside \(B_{2R}(c)\);
\item the extension of the transformed coefficients outside \(B_{2R}(c)\);
\item the extension of the initial datum outside \(B_{2R}(c)\).
\end{enumerate}

More precisely, any two such choices produce the same solution in a common
cone
\[
    \{(t,y):0<t<T,\ |y-c|<R-Lt\},
\]
where \(L\) is the maximum of the two corresponding propagation speeds.
Pulling this common local solution back by \(X\) defines a local solution of
\eqref{eq:general-main} near \(x_0\).

For \(BV\) base data, the chartwise quotient variables have strong one-sided
traces, and the normalized traces
\eqref{eq:general-normalized-projected-traces} satisfy
\[
    (r_\Gamma^-,r_\Gamma^+)
    \in
    \mathcal G_{\mathrm{vv}}^\nu
\]
at almost every regular interface point contained in the cone.
\end{proposition}

\begin{proof}
Consider two global extensions which coincide on \(B_{2R}(c)\), and whose
initial data also coincide there. In the region where the two coefficients
coincide, the local Kato inequality contains no coefficient-difference term.
Testing it with a shrinking-cone cutoff gives
\[
\int_{B_{R-Lt}(c)}
    |w_1(t,y)-w_2(t,y)|
\,dy
\le
C_t
\int_{B_R(c)}
    |w_1(0,y)-w_2(0,y)|
\,dy,
\]
where \(L\) is a common propagation speed. The right-hand side vanishes, so
the two solutions agree in the common cone.

For \(BV\) base data, the trace and germ assertions follow from the
flat-interface theorem and Lemma~\ref{lem:general-chart-covariance}. For
general local data, independence follows by \(BV\)-approximation and local
\(L^1\)-stability.
\end{proof}

At points of \(\mathbb R^d\setminus\Gamma\), the analogous local
construction uses the closed heterogeneous Kruzhkov solution operator for the
single smooth branch.

\subsection{Compatibility on chart overlaps}
\label{subsec:general-overlap}

\begin{proposition}[Compatibility of local solvers]
\label{prop:general-overlap-compatibility}
Let \(u_1\) and \(u_2\) be two local solutions constructed in possibly
different charts. Assume that they are based at the same initial time and
have the same initial datum on the base of a cone compactly contained in the
overlap of their physical domains. Then
\[
    u_1=u_2
\]
almost everywhere in a smaller cone.

The same conclusion holds for local solutions obtained from initial data in
the \(L^1_{\mathrm{loc}}\)-closure of \(BV\).
\end{proposition}

\begin{proof}
We first consider \(BV\) base data. Away from the interface, the conclusion
follows from the standard local Kato inequality.

Suppose that the overlap meets \(\Gamma_{\mathrm{reg}}\). Choose a common
orientation of the physical interface and write both solutions in one
flattening chart. By Lemma~\ref{lem:general-chart-covariance}, their
normalized trace pairs belong to the same intrinsic germ
\[
    \mathcal G_{\mathrm{vv}}^\nu(x).
\]
If one of the original charts used the opposite orientation, the two
components are first interchanged according to
\eqref{eq:general-orientation-reversal}.

In the chosen chart, the interface contribution in the doubled Kato
inequality is
\[
\begin{aligned}
\sigma(x)
\Bigl[
&Q_+^\nu
\bigl(
    x;
    r_{1,\Gamma}^+,
    r_{2,\Gamma}^+
\bigr)
\\
&-
Q_-^\nu
\bigl(
    x;
    r_{1,\Gamma}^-,
    r_{2,\Gamma}^-
\bigr)
\Bigr].
\end{aligned}
\]
The expression in brackets is nonpositive by the
\(L^1\)-dissipativity of
\(\mathcal G_{\mathrm{vv}}^\nu(x)\), and \(\sigma(x)>0\).
The interface term therefore has the correct sign.

A shrinking-cone cutoff, together with the local Lipschitz estimate for the
divergences, gives
\[
    \int_{B_{R-Lt}}
        |u_1(t,x)-u_2(t,x)|
    \,dx
    \le
    C_t
    \int_{B_R}
        |u_1(0,x)-u_2(0,x)|
    \,dx.
\]
The right-hand side is zero.

For general base data, choose a common sequence of \(BV\)-approximations on
the base of the cone. The corresponding \(BV\) local solutions agree by the
first part, and local stability permits passage to the limit.
\end{proof}

\subsection{Patching and continuation}
\label{subsec:general-patching}

The preceding propositions show that the local solvers form a compatible
family on
\[
    (0,\infty)
    \times
    \bigl(\mathbb R^d\setminus\Gamma_p\bigr).
\]
To organize the construction, set
\begin{equation}
\label{eq:general-compact-exhaustion}
    E_n
    :=
    \left\{
        x\in\mathbb R^d:
        |x|\le n,
        \quad
        \operatorname{dist}(x,\Gamma_p)\ge\frac1n
    \right\}.
\end{equation}
Then
\[
    E_n\Subset\operatorname{int}E_{n+1},
    \qquad
    \bigcup_{n=1}^\infty E_n
    =
    \mathbb R^d\setminus\Gamma_p.
\]

For each \(n\), choose finitely many pairs of chart neighborhoods
\[
    V_{n,j}\Subset U_{n,j},
    \qquad
    j=1,\ldots,J_n,
\]
such that the \(V_{n,j}\) cover \(E_n\), while the local solvers are
constructed in the larger sets \(U_{n,j}\). After decreasing the cone
heights, there exists \(h_n>0\) such that, for every
\(0\le t\le h_n\), the spatial sections of the corresponding local cones
still cover the sets \(V_{n,j}\).

Starting from \(u_0\), construct the local solutions in all charts. By
Proposition~\ref{prop:general-overlap-compatibility}, they agree on overlaps
and hence define a solution locally in space and time on
\[
    \mathbb R^d\setminus\Gamma_p.
\]
Strong \(L^1_{\mathrm{loc}}\)-continuity of the local solution operators
permits restarting at any positive time. The semigroup property and overlap
compatibility show that the continuation is independent of the chosen
subdivision of the time interval.

Applying this construction on the compact exhaustion and using a diagonal
argument gives a function
\begin{equation}
\label{eq:general-solution-away-singular-set}
    u
    \quad\text{on}\quad
    (0,\infty)
    \times
    \bigl(\mathbb R^d\setminus\Gamma_p\bigr).
\end{equation}
The construction is independent of the atlas, the order in which the charts
are used, the auxiliary global extensions, and the restarting times. This is
the same compact-exhaustion and cone-patching argument as in
\cite[Section~2]{Mitrovic2025Nondegenerate}; the only change is the local
flat-interface solver.

\begin{remark}[Strong and closed interface admissibility]
\label{rem:general-strong-versus-closed}
When the datum at the base of a cone is \(BV\), the local solution has strong
projected traces satisfying
\eqref{eq:general-intrinsic-germ-condition}. At a later restarting time, the
new datum need not be \(BV\). Continuation is therefore made with the
\(L^1_{\mathrm{loc}}\)-closure of the local flat-interface solution
operator.

This formulation does not require an unproved stability statement for strong
projected traces under arbitrary \(L^1_{\mathrm{loc}}\)-closure. Any local
solution which already possesses strong normalized projected traces satisfying
\eqref{eq:general-intrinsic-germ-condition} is unique by the local Kato
inequality and therefore agrees with the corresponding closed local
solution.
\end{remark}

\subsection{Removal of the interaction set}
\label{subsec:general-removal}

The condition
\[
    \mathcal H^{d-1}(\Gamma_p)=0
\]
allows the singular interaction set to be removed from the weak and Kato
formulations.

\begin{lemma}[A \(W^{1,1}\)-cutoff around \(\Gamma_p\)]
\label{lem:general-small-set-cutoff}
Let \(K_0\Subset\mathbb R^d\). For every \(\delta>0\), there exists
\[
    \omega_\delta\in C^\infty(\mathbb R^d)
\]
such that
\[
    0\le\omega_\delta\le1,
\]
\(\omega_\delta=0\) in a neighborhood of
\(\Gamma_p\cap K_0\), and \(\omega_\delta=1\) outside a slightly larger
neighborhood, with
\begin{equation}
\label{eq:general-small-set-cutoff-estimate}
\left|
    \left\{
        x\in K_0:
        \omega_\delta(x)\neq1
    \right\}
\right|
+
\int_{K_0}
    |\nabla\omega_\delta(x)|
\,dx
\le\delta.
\end{equation}
\end{lemma}

\begin{proof}
For every \(\eta>0\), the compact set
\(\Gamma_p\cap K_0\) can be covered by finitely many balls
\[
    B(x_i,r_i),
    \qquad
    r_i<\eta,
\]
such that
\[
    \sum_i r_i^{d-1}<\eta.
\]
Choose
\[
    \theta_i\in C_c^\infty(B(x_i,2r_i)),
\]
equal to one on \(B(x_i,r_i)\), with
\[
    |\nabla\theta_i|\le \frac{C}{r_i}.
\]
Set
\[
    \omega_\delta
    :=
    \prod_i(1-\theta_i).
\]
Then
\[
    \int_{\mathbb R^d}
        |\nabla\omega_\delta|
    \,dx
    \le
    C\sum_i r_i^{d-1},
\]
while
\[
    |\{\omega_\delta\neq1\}|
    \le
    C\sum_i r_i^d
    \le
    C\eta\sum_i r_i^{d-1}.
\]
Choosing \(\eta\) sufficiently small proves
\eqref{eq:general-small-set-cutoff-estimate}.
\end{proof}

\begin{proposition}[Removability of \(\Gamma_p\)]
\label{prop:general-removability}
Let \(u\) be the solution constructed on
\[
    (0,\infty)
    \times
    \bigl(\mathbb R^d\setminus\Gamma_p\bigr).
\]
After redefining \(u\) arbitrarily on
\((0,\infty)\times\Gamma_p\), it satisfies the weak formulation of
\eqref{eq:general-main} on all of
\((0,\infty)\times\mathbb R^d\).

The same cutoff argument removes \(\Gamma_p\) from the global Kato
inequality for two locally admissible solutions.
\end{proposition}

\begin{proof}
Let
\[
    \varphi
    \in
    C_c^\infty
    \bigl([0,\infty)\times\mathbb R^d\bigr),
\]
and let \(K_0\) contain the spatial projection of its support. Use
\[
    \varphi(t,x)\omega_\delta(x)
\]
as a test function in the weak formulation valid away from \(\Gamma_p\).
All terms not containing \(\nabla\omega_\delta\) converge by dominated
convergence. The remaining term is bounded by
\[
    C\|\varphi\|_{L^\infty}
    \int_{K_0}
        |\nabla\omega_\delta(x)|
    \,dx,
\]
and therefore tends to zero.

For the Kato inequality, the additional term is
\[
    \operatorname{sgn}(u-v)
    \bigl(
        \mathfrak f(x,u)-\mathfrak f(x,v)
    \bigr)
    \cdot\nabla\omega_\delta,
\]
which is estimated in the same way.
\end{proof}

\subsection{Admissible solutions}
\label{subsec:general-admissible-solutions}

\begin{definition}[Admissible solution for a general interface family]
\label{def:general-admissible-solution}
A function
\[
    u
    \in
    L^\infty
    \bigl((0,\infty)\times\mathbb R^d;K\bigr)
    \cap
    C
    \bigl([0,\infty);L^1_{\mathrm{loc}}(\mathbb R^d)\bigr)
\]
is called an admissible solution of
\eqref{eq:general-main} if the following conditions hold.

\smallskip
\noindent
\emph{(i) Weak formulation.}
The equation holds in
\[
    \mathcal D'
    \bigl((0,\infty)\times\mathbb R^d\bigr),
\]
and
\[
    u(0,\cdot)=u_0
    \quad\text{in }L^1_{\mathrm{loc}}(\mathbb R^d).
\]

\smallskip
\noindent
\emph{(ii) Entropy inequalities away from the interfaces.}
On every open set compactly contained in
\(\mathbb R^d\setminus\Gamma\), the function \(u\) satisfies the
heterogeneous Kruzhkov entropy inequalities for the corresponding smooth
flux branch.

\smallskip
\noindent
\emph{(iii) Local interface admissibility.}
For every
\[
    \tau\ge0,
    \qquad
    x_0\in\Gamma_{\mathrm{reg}},
\]
the pullback of \(u\) to a flattening chart around \(x_0\) agrees, in a
sufficiently small forward cone based at time \(\tau\), with the closed local
flat-interface solution operator started from the local datum
\(u(\tau,\cdot)\).

No additional condition is imposed on \(\Gamma_p\).
\end{definition}

\begin{remark}
By local \(L^1\)-stability, condition \emph{(iii)} is equivalent to local
approximability by flat-interface solutions with \(BV\) base data. If the
quotient variables of \(u\) possess strong one-sided traces, then
condition \emph{(iii)} is also equivalent to
\[
    \bigl(
        r_\Gamma^-(t,x),
        r_\Gamma^+(t,x)
    \bigr)
    \in
    \mathcal G_{\mathrm{vv}}^\nu(x)
\]
for a.e.
\((t,x)\in(0,\infty)\times\Gamma_{\mathrm{reg}}\).
\end{remark}

\subsection{The general-interface theorem}
\label{subsec:general-main-theorem}

\begin{theorem}[Well-posedness for general interface geometry]
\label{thm:general-interface}
Assume
\eqref{eq:general-interface-decomposition}--%
\eqref{eq:general-local-graph},
the local two-branch representation
\eqref{eq:general-local-flux},
the regularity and invariant-region conditions
\eqref{eq:general-branch-regularity}--%
\eqref{eq:general-invariant},
and
\[
    \mathcal H^{d-1}(\Gamma_p)=0.
\]
Assume moreover that the interface and smooth-flux charts are locally finite
on \(\mathbb R^d\setminus\Gamma_p\), and that the local branch
representations agree on overlaps, up to interchange of the two sides.

Then, for every
\[
    u_0\in BV(\mathbb R^d;K),
\]
there exists a unique admissible solution of
\eqref{eq:general-main} in the sense of
Definition~\ref{def:general-admissible-solution}.

If \(u\) and \(v\) are admissible solutions with initial data \(u_0\) and
\(v_0\), then, for every \(T,R>0\), there exist
\[
    \overline R>R,
    \qquad
    C_{T,R}>0,
\]
depending only on the local structural bounds in the corresponding domain of
dependence, such that
\begin{equation}
\label{eq:general-interface-stability}
\operatorname*{ess\,sup}_{t\in(0,T)}
\int_{B_R}
    |u(t,x)-v(t,x)|
\,dx
\le
C_{T,R}
\int_{B_{\overline R}}
    |u_0(x)-v_0(x)|
\,dx.
\end{equation}
\end{theorem}

\begin{proof}
The local interface and smooth-flux solution operators are provided by
Proposition~\ref{prop:general-local-solver} and the ordinary heterogeneous
Kruzhkov theory. Proposition~\ref{prop:general-overlap-compatibility}
shows that they agree on overlaps. The patching and continuation argument of
Subsection~\ref{subsec:general-patching} therefore gives a solution on
\[
    (0,\infty)
    \times
    \bigl(\mathbb R^d\setminus\Gamma_p\bigr).
\]
Proposition~\ref{prop:general-removability} extends the weak formulation
across \(\Gamma_p\), and hence produces an admissible solution on the whole
space.

For uniqueness, apply the local Kato inequality in every regular interface
chart and every smooth-flux chart. At regular interface points, the interface
contribution is nonpositive because both local solutions belong to the same
closed local solver; for \(BV\) approximations this follows from the maximal
\(L^1\)-dissipativity of the intrinsic projected germ, and the inequality
passes to the closure by local \(L^1\)-stability.

The local inequalities patch to a global Kato inequality away from
\(\Gamma_p\). Lemma~\ref{lem:general-small-set-cutoff} removes
\(\Gamma_p\). Standard finite-speed cutoffs and Gronwall's inequality give
\eqref{eq:general-interface-stability}. Equal initial data imply \(u=v\).
\end{proof}

\begin{corollary}[Extension beyond \(BV\) initial data]
\label{cor:general-extension-l1}
The solution map extends uniquely by local \(L^1\)-continuity to every
\[
    u_0
    \in
    L^1_{\mathrm{loc}}(\mathbb R^d)
    \cap
    L^\infty(\mathbb R^d),
    \qquad
    u_0(x)\in K
    \quad\text{for a.e. }x.
\]
The extended solution is the strong \(L^1_{\mathrm{loc}}\)-limit of the
solutions corresponding to \(BV(\mathbb R^d;K)\)-approximations of \(u_0\),
is independent of the approximating sequence, and satisfies
\eqref{eq:general-interface-stability}.
\end{corollary}

\begin{proof}
Choose
\[
    u_0^n\in BV(\mathbb R^d;K),
    \qquad
    u_0^n\longrightarrow u_0
    \quad\text{in }L^1_{\mathrm{loc}}(\mathbb R^d),
\]
and let \(u^n\) be the corresponding admissible solutions.
Estimate~\eqref{eq:general-interface-stability} shows that \((u^n)\) is
Cauchy in
\[
    L^\infty_{\mathrm{loc}}
    \bigl(
        [0,\infty);
        L^1_{\mathrm{loc}}(\mathbb R^d)
    \bigr).
\]
The same estimate proves independence of the approximation and stability of
the limit.
\end{proof}

    \section*{Acknowledgements}

I would like to express my sincere gratitude to Professor Kenneth H. Karlsen,
who introduced me to the problem of scalar conservation laws with
discontinuous flux during my postdoctoral stay in Trondheim nearly twenty
years ago. I am grateful for the many discussions, suggestions, and insights
he has shared with me over the years, as well as for his extraordinary
patience and generosity.

 \end{document}